\documentclass[a4paper, 10pt, DIV10, headinclude=false, footinclude=false]{scrartcl}

\usepackage{amsthm, amsmath, amssymb, amsfonts, esint}
\usepackage[utf8]{inputenc}
\usepackage[english]{babel}
\usepackage{url}
\usepackage{mathtools}

\usepackage{subcaption}
\usepackage{tikz}
\usepackage{pgfplots}
\usepackage{pgfplotstable}
\usepackage{booktabs}
\usepackage{algorithm,algpseudocode}

\numberwithin{figure}{section}
\numberwithin{table}{section}
\numberwithin{equation}{section}

\newenvironment{abstr}[1]{ \vspace{.05in}\footnotesize
	\parindent .2in
	{\upshape\bfseries #1. }\ignorespaces}{\par\vspace{.1in}}
\newenvironment{Abstract}{\begin{abstr}{Abstract}}{\end{abstr}}
\newenvironment{keywords}{\begin{abstr}{Key words}}{\end{abstr}}
\newenvironment{AMS}{\begin{abstr}{AMS subject classifications}}{\end{abstr}}



\newtheorem{theorem}{Theorem}[section]

\newtheorem{corollary}[theorem]{Corollary}

\theoremstyle{definition}

\newtheorem{example}[theorem]{Example}
\newtheorem{remark}[theorem]{Remark}

\usepackage{todonotes}

\newcommand{\nz}{\mathbb{N}}
\newcommand{\rz}{\mathbb{R}}
\newcommand{\calC}{\mathcal{C}}
\newcommand{\calN}{\mathcal{N}}
\newcommand{\calQ}{\mathcal{Q}}
\newcommand{\calT}{\mathcal{T}}

\newcommand{\fraka}{\mathfrak{a}}
\newcommand{\frakb}{\mathfrak{b}}

\newcommand{\Vf}{V^\mathrm{f}}
\newcommand{\Vms}{V_{H,m}^\mathrm{ms}}

\allowdisplaybreaks[4]

\begin{document}
	
	\title{An offline-online strategy for multiscale problems with random defects\thanks{This work was initiated while BV enjoyed the kind hospitality of Chalmers University, Gothenburg.
	AM is funded by the Swedish Research Council and the G\"{o}ran Gustafsson
foundation for Research in Natural Sciences and Medicine.
	BV is funded by the German Research Foundation (DFG) -- Project-ID 258734477 -- SFB 1173 and Klaus-Tschira foundation as well as by the Federal Ministry of Education and Research (BMBF) and the Baden-W\"urttemberg Ministry of Science as part of the Excellence Strategy of the German Federal and State Governments.}%
	}
	\author{Axel M\aa lqvist\footnotemark[2] \and Barbara Verf\"urth\footnotemark[3]}
	\date{}
	\maketitle

	\renewcommand{\thefootnote}{\fnsymbol{footnote}}
	\footnotetext[2]{Department of Mathematical Sciences, Chalmers University of Technology and University of Gothenburg, 412 96 G\"oteborg, Sweden
%\email{axel@chalmers.se}
}

\footnotetext[3]{Institut f\"ur Angewandte und Numerische Mathematik, Karlsruher Institut f\"ur Technologie, Englerstr. 2, 76131 Karlsruhe, Germany
%\email{barbara.verfuerth@kit.edu}
}

	\renewcommand{\thefootnote}{\arabic{footnote}}
	
	\begin{Abstract}
	In this paper, we propose an offline-online strategy based on the Localized Orthogonal Decomposition (LOD) method for elliptic multiscale problems with randomly perturbed diffusion coefficient. We consider a periodic deterministic coefficient with local defects that occur with probability $p$.
	The offline phase pre-computes entries to global LOD stiffness matrices on a single reference element (exploiting the periodicity) for a selection of defect configurations. Given a sample of the perturbed diffusion the corresponding LOD stiffness matrix is then computed by taking linear combinations of the pre-computed entries, in the online phase. Our computable error estimates show that this yields a good coarse-scale approximation of the solution for small $p$, which is illustrated by extensive numerical experiments.  This makes the proposed technique attractive already for moderate sample sizes in a Monte Carlo simulation.
	\end{Abstract}
	
	\vspace*{-2ex}
	
	\vspace*{-2ex}

\section{Introduction}
Many modern materials include some fine composite structure to achieve enhanced properties. Examples include fiber reinforced structures in mechanics as well as mechanical, acoustic or optical metamaterials.
The materials are often highly structured, but mistakes in the fabrication process lead to defects. A major question is the robustness of the desired material properties under such defects.
Mathematically speaking, we are interested in the solution of partial differential equations (PDEs) with multiscale, randomly perturbed coefficients. 

In this paper, we study the following elliptic multiscale problem: Find $u:D\to \mathbb{R}$ such that
\begin{equation}\label{eq:pbstrong}
-\nabla \cdot (A(x)\nabla u(x)) = f(x)\qquad \text{in}\quad D
\end{equation}
with suitable boundary conditions.
Here, $D$ is a spatial domain in $\rz^d$ and $f\in L^2(D)$.
The multiscale coefficient $A\in L^\infty(D, \rz)$ is a sample of a randomly perturbed coefficient. More precisely we assume that $A$ is a realization of the form 
\begin{equation}\label{eq:weaklyrandom}
A(x,\omega) = A_{\varepsilon}(x)+b_{p,\varepsilon}(x, \omega)B_{\varepsilon}(x),
\end{equation}
where $A_{\varepsilon}, B_{\varepsilon}$ are deterministic multiscale coefficients and $b_{p,\varepsilon}(x, \cdot)$ is a Bernoulli law with probability $p$, cf.~\cite{AnaLb12}.
Detailed assumptions on the problem data and the form of $A$ are given in Section~\ref{sec:setting} below.
Two important examples of this setup are illustrated in Figure~\ref{fig:randcoeffs}. 
On the left, $A$ is generated from a constant $A_{\varepsilon}$ by introducing square spots with length $\varepsilon$ and probability $p$. 
On the right, $A_{\varepsilon}$ is made of a background value and periodic square inclusion repeating with a periodicity length $\varepsilon$. 
In this case, $A$ is generated by randomly setting some of the inclusions to the background value, thus ``erasing'' them.
The considered model of so-called weakly random coefficients, characterized by small values of $p$, also covers other defect possibilities of the inclusions like a change of value, a (fixed) shift or a (fixed) change of the geometry.

\begin{figure}
	\centering
	\includegraphics[width=0.47\textwidth, trim=19mm 8mm 20mm 10mm, clip=true]{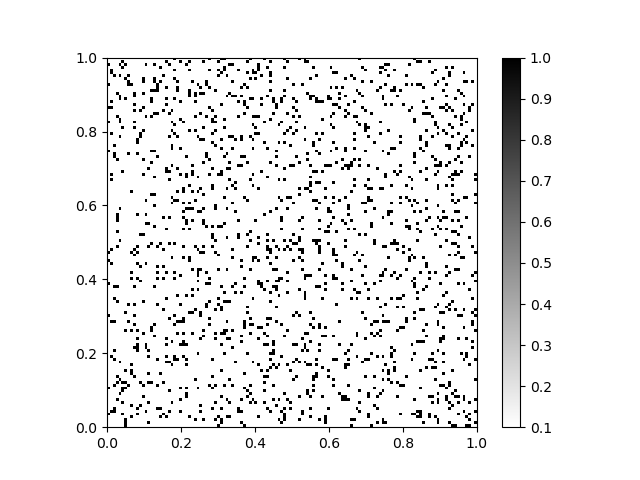}%
	\hspace{2ex}%
	\includegraphics[width=0.47\textwidth, trim=19mm 8mm 20mm 10mm, clip=true]{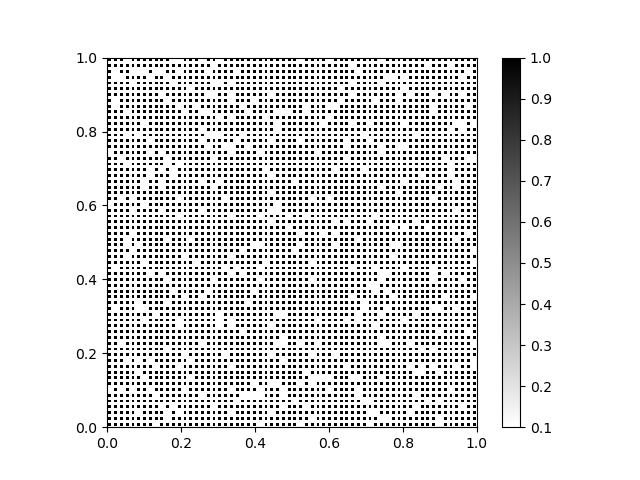}
	\caption{Two examples for weakly random coefficients: random checkerboard ($\varepsilon=2^{-7}, p=0.1$, left) and periodic inclusions with random ``erasure'' ($\varepsilon=2^{-6}, p=0.1$, right)}
	\label{fig:randcoeffs}
\end{figure}

In the context of materials with defects, one is interested in extracting statistical information about the solution $u$.
There are many different uncertainty quantification techniques for PDEs with random coefficients, see, e.g., \cite{BabNT07,GunWZ14,LorPS14} for overviews.
In the following, we focus on Monte Carlo (MC)-type approaches such as Quasi Monte Carlo or Multilevel Monte Carlo (MLMC) \cite{BarSZ11,CliGST11,TecSGU13,ElfHM16}.
Hence, we are interested in (approximate) solutions to~\eqref{eq:pbstrong} for many samples (i.e., realizations) of $A$.
Due to the multiscale nature of $A$, standard discretization schemes like the finite element method would require the mesh to resolve all fine-scale features. Consequently, already the computation of a few solutions to~\eqref{eq:pbstrong} becomes prohibitively costly.
In contrast, computational multiscale methods such as the Localized Orthogonal Decomposition (LOD) \cite{HenP13,MalP14,MalP20,AbdH15} yield faithful coarse-scale approximations with feasible effort after pre-computation of a generalized finite element basis. However, these basis functions incorporate knowledge about the multiscale coefficient $A$ and thus, need be constructed anew for each realization in general. This makes it difficult to apply multiscale methods to stochastic problems. Recently there have been several attempts to circumvent this difficulty, for example the combination of the Multiscale Finite Element Method with Multilevel Monte Carlo  \cite{EfeKL15} or low-rank approximation \cite{OuLJ20}, the multiscale data-driven stochastic method \cite{ZhaCH15}, an approach for a quasi-local homogenized coefficient \cite{FisGP19,GalP19}, and a sparse compression of the expected solution operator \cite{FeiP20}.
In the context of the aforementioned LOD, the recent works \cite{HelM19,HelKM20} deal with rare defects. They propose to compute the multiscale basis for the unperturbed deterministic coefficient and to update this basis only locally for each particular sample.
More precisely, given a sample coefficient $A$, a computable  error indicator shows whether the pre-computed basis is sufficiently good or if a new (improved) basis needs to be computed for that particular sample.
If $p$ is small enough this technique becomes competitive.
A pre-computed, deterministic multiscale basis is also the key idea of the Multiscale Finite Element approach of \cite{LbLT14}.
An asymptotic expansion in the random variable is used for the numerical analysis of \cite{LbLT14} as well as to approximate the effective coefficient of stochastic homogenization  \cite{AnaLb11,AnaLb12,LbT12,Lb14} or to reduce its variance \cite{BlaLbL16,Lb14}.

This contribution is an attempt to make multiscale methods useful for a wider range of random problems. Instead of having one reference coefficient, as in \cite{HelKM20}, we build a ``basis'' of reference coefficients $\{A_i\}_{i=0}^N$ and pre-compute and store the corresponding LOD basis functions $\{\lambda-\mathcal{C}(A_i)\lambda\}_{i=0}^N$, $\lambda$ being the finite element basis function and $\mathcal{C}(B)$ the LOD correction based on coefficient $B$. This is done in the offline phase. We consider problems with periodic structure so that the same set of precomputed basis functions can be used in the entire domain. Given samples of the form $A=\sum_{i=0}^N \mu_i A_i$ we let $\sum_{i=0}^N \mu_iA_i\nabla (\mathcal{C}(A_i)\lambda)$ approximate $A\nabla (\mathcal{C}(A)\lambda)$ in the online phase. This allows for rapid assembly of the LOD stiffness matrices and thereby solution of the problem given samples $A$. We demonstrate theoretically that the error $A\nabla(\mathcal{C}(A)\lambda)-\sum_{i=0}^N \mu_iA_i\nabla(\mathcal{C}(A_i)\lambda)$ is small for small $p$ in one dimension and provide a computable error indicator of this quantity in higher dimensions. We also present numerical experiments for a large variety of configurations of the diffusions which show relative root mean square errors up to $3\%$ for defect probabilities of $p=0.1$ and below in settings with moderate contrast. We compare with \cite{HelKM20} and show a substantial improvement. The strategy pays off in a Monte Carlo setting with a moderate sample size.

The paper is organized as follows. In Section~\ref{sec:setting}, we formulate the model problem and detail the form of $A$.
In Section~\ref{sec:newPGLOD}, we review the Petrov-Galerkin Localized Orthogonal Decomposition (PG-LOD) and introduce our new offline-online strategy.
A priori error estimates for the new method are presented in Section~\ref{sec:analysis}.
We discuss several implementation details with a focus on computational efficiency in Section~\ref{sec:impl}.
Extensive numerical experiments in Section~\ref{sec:numexp} showcase the attractive properties of the method and also illustrate our theoretical findings.

\section{Problem formulation}
\label{sec:setting}
In the following, we detail the setting associated with~\eqref{eq:pbstrong}. We first pose the problem for a fixed event $\omega$ in a probability space $\Omega$ and then discuss the specific form of randomness in the coefficient.
By slight abuse of notation, we will omit the random variable in the following exposition for a fixed, but arbitrary sample.

\subsection{Model problem}\label{sec:setting:pb}
For simplicity we let $D=[0,1]^d\subset \rz^d$ be the unit cell.  
We assume that $f\in L^2(D)$ and that the realization $A\in L^\infty(D, \rz)$ is uniformly bounded and elliptic, i.e.,
\begin{equation}\label{eq:boundsA}
0<\alpha:=\mathrm{ess} \inf_{x\in D} A(x), \qquad \infty>\beta:=\mathrm{ess} \sup_{x\in D} A(x).
\end{equation}
We introduce a function space $V$ where we seek a solution of~\eqref{eq:pbstrong} in weak form.
In this paper we mainly consider a conforming finite element space
\[V:=V_h\subset H^1_{\#, 0}(D)=\{v\in H^1(D)\, |\, v \text{ is periodic}, \int_Dv =0\}\]
defined on a computational mesh $\calT_h$ which resolves the variations in $A$.
Further, we assume $\calT_h$ to be a conforming (i.e., without hanging nodes and edges) and shape regular quadrilateral mesh that can additionally be wrapped into a mesh on the  torus without hanging nodes or edges.
However, it is also possible to choose $V=H^1_{\#,0}(D)$ and the following analysis will still go through.
The weak form of~\eqref{eq:pbstrong} reads as follows: find $u\in V$ such that
\begin{equation}\label{eq:pbweak}
\fraka(u,v)=F(v)\qquad\text{for all}\quad v\in V,
\end{equation}
where 
\[\fraka(u,v):=\int_D A(x)\nabla u(x)\cdot \nabla v(x)\, dx,\qquad F(v):=\int_D f(x)v(x)\, dx.\]
Due to the constraint $\int_D u(x,\omega)\, dx=0$, existence and uniqueness of a solution $u$ is guaranteed by the Lax-Milgram lemma.
We will frequently use the energy norm $\|\cdot \|_A:=\fraka(\cdot, \cdot)^{1/2}$ and its restriction $\|\cdot \|_{A, S}$ to a subdomain $S\subset D$ in the following. Further, $(\cdot, \cdot)_S$ denotes the usual $L^2$-scalar product on $S$ where we omit the subscript if $S=D$.

\begin{remark}
	We consider periodic boundary conditions and box-type domains in this paper to fully exploit the underlying structure in $A$.
	However, with additional computational effort, Dirichlet or Neumann boundary conditions as well as more general Lipschitz domains can be treated, see Remark~\ref{rem:bdrycondoffline}.
	In particular, the error analysis is not restricted to periodic boundary conditions or box-type domains.
\end{remark}

\subsection{Randomly perturbed coefficients}\label{subsec:weaklyrandom}
As mentioned above, we are interested in solving~\eqref{eq:pbweak} for many different (random) choices of $A$.
We now give more details on the form~\eqref{eq:weaklyrandom} of $A(x, \omega)$, similar to the ``weakly'' random setting of \cite{AnaLb12}.
We assume that $A_\varepsilon$ and $B_\varepsilon$ are deterministic multiscale coefficients.
More specifically, we let $A_\varepsilon(x)=A_\mathrm{per}(x/\varepsilon)$ where $A_\mathrm{per}$ is $1$-periodic  and we assume that $\varepsilon =1/n$ with $n\in \mathbb{N}$, $n\gg 1$. The same form $B_\varepsilon(x)=B_\mathrm{per}(x/\varepsilon)$ is assumed for $B_\varepsilon$.
The periodicity assumption -- together with the box-type domain -- is imposed for efficiency reasons of our method, where we emphasize that generalizations are possible, see Remark~\ref{rem:structoffline}.

The deterministic coefficients are assumed to satisfy spectral bounds similar as~\eqref{eq:boundsA}, i.e.,
\begin{equation}\label{eq:boundsAper}
0<\alpha\leq\mathrm{ess} \inf_{x\in D} A_\mathrm{per}(x), \qquad \infty>\beta\geq\mathrm{ess} \sup_{x\in D} A_\mathrm{per}(x).
\end{equation}
and
\begin{equation}\label{eq:boundsAperCper}
0<\alpha\leq\mathrm{ess} \inf_{x\in D} (A_\mathrm{per}(x)+B_\mathrm{per}(x)), \qquad \infty>\beta\geq\mathrm{ess} \sup_{x\in D} (A_\mathrm{per}(x)+B_\mathrm{per}(x)).
\end{equation}
The random character of $A(x,\omega)$ in~\eqref{eq:weaklyrandom} is encoded in $b_{p,\varepsilon}(x,\omega)$ for which we assume
\begin{equation}\label{eq:randomness}
b_{p, \varepsilon}(x,\omega) =\sum_{j \in I}\chi_{\varepsilon(j+Q)}(x)\hat b_p^j(\omega).
\end{equation}
Here, $\chi$ denotes the characteristic function, $Q\subseteq [0,1]^d$ and $I:=\{k\in \mathbb{Z}^d\, |\, \varepsilon(k+Q)\subset D\}$.
Finally, $\hat b_p^j$ are independent random variables adhering to a Bernoulli distribution with probability $p$, i.e., $\hat b_p^j=0$ with probability $1-p$ and $\hat b_p^j=1$ with probability $p$.
Clearly, for $p\to 0$, the defects/perturbations encoded in $b_{p,\varepsilon}$ become rare events.
This choice of $b_{p,\varepsilon}$ together with the assumptions~\eqref{eq:boundsAper}--\eqref{eq:boundsAperCper} guarantee that each realization $A(x, \omega)$ satisfies~\eqref{eq:boundsA}.
We close this section by giving two examples for this setting, namely the formal definition of the coefficients depicted in Figure~\ref{fig:randcoeffs} above.

\begin{example}[Random checkerboard]\label{ex:randcheck}
	Recall that the coefficient in Figure~\ref{fig:randcoeffs}, left, is piece-wise constant on a square mesh $\calT_\varepsilon$. On each square element, the value of $A$ is picked randomly as either $\alpha$ with probability $1-p$ or as $\beta$ with probability $p$.
	This can be described in the form~\eqref{eq:weaklyrandom} and~\eqref{eq:randomness} with $Q=[0,1]^d$, $A_\mathrm{per}=\alpha$ and $B_\mathrm{per}=\beta-\alpha$.
\end{example}

\begin{example}[Periodic coefficient with random defects]\label{ex:randdef}
	Realizations as depicted in Figure~\ref{fig:randcoeffs}, right, can be formalized in the following way.
	We define $A_\mathrm{per}: [0,1]^d\to \rz$ via
	\begin{equation*}
	A_\mathrm{per}(y):=\begin{cases} \beta\quad y\in [0.25, 0.75]^d,\\
	\alpha \quad \text{else}.
	\end{cases}
	\end{equation*}
	Further, we pick $Q=[0.25, 0.75]^d$ and 
	$B_\mathrm{per} =\alpha-\beta$. Note that $B_{\mathrm{per}}$ is only added in the shifted and scaled copies of $Q$.
	Clearly, any other value $\tilde{\beta}\in [\alpha, \beta]$ as defect can be modeled by the choice $B_\mathrm{per} =\tilde \beta-\alpha$.
	Even a value $0<\tilde{\beta}\notin[\alpha, \beta]$ is possible for the defects by changing the spectral bounds.
	If the defect changes the shape of the inclusion, we have to define $Q$ and $B_\mathrm{per}$ accordingly.
	For example, imagine that a defect means that the value $\beta$ is taken in (scaled and shifted copies of) $[0.75, 1]^d$. Then, $A_\mathrm{per}$ is left unchanged, we set $Q=[0,1]^d$ and define $B_\mathrm{per}:Y\to \rz$ via 
	\begin{equation*}
	B_\mathrm{per}(y):=\begin{cases} \alpha-\beta\quad y\in [0.25, 0.75]^d,\\
	\beta-\alpha \quad y\in [0.75, 1]^d,\\
	0 \qquad\quad \text{else}.
	\end{cases}
	\end{equation*} 
\end{example}

\section{Offline-online strategy for the PG-LOD}
\label{sec:newPGLOD}
In this section, we first review the Petrov-Galerkin Localized Orthogonal Decomposition (PG-LOD) in Sections~\ref{subsec:fenotation} and~\ref{subsec:LOD} and then present our new offline-online strategy in Section~\ref{subsec:LODnew}.
Connections with and comparison to homogenization approaches are briefly discussed in Section~\ref{subsec:hom}.
Throughout this paper, we further use the notation $a\lesssim b$ if $a\leq cb$ with a generic constant $c$ that only depends on the shape regularity of the mesh, the domain $D$, or the space dimension $d$.

\subsection{Preliminaries and notation}\label{subsec:fenotation}
Let $\calT_H$ be a coarse, shape regular, quasi-uniform and conforming quadrilateral mesh of the domain $D$.
We further assume that $\calT_H$ can be wrapped into a conforming mesh of the torus, i.e., no hanging nodes and edges occur over the periodic boundary.
Let $H=\max_{T\in \calT_H}\mathrm{diam} T$ denote the mesh size.
The standard lowest-order finite element space on $\calT_H$ is given as
\[V_H:=H^1_{\#, 0}(D)\cap \calQ_1(\calT_H),\]
where $\calQ_1(\calT_H)$ denotes the space of $\calT_H$-piecewise polynomials of coordinate degree at most $1$.
Note that functions in $\calQ_1(\calT_H)$ may be discontinuous.
We further assume that the fine mesh $\calT_h$ is a refinement of $\calT_H$ such that the finite element spaces are nested as $V_H\subset V_h$.

We further introduce a notion of element patches. 
For an arbitrary subdomain $S\subset D$ and $m\in \nz_0$, we define patches $U_m(S)\subset D$ inductively as
\[U_0(S)=S, \qquad U_{m+1}(D)=\bigcup\{T\in \calT_H\,|\, \overline{U_m(S)}\cap \overline{T}\neq \emptyset\}.\]
In this definition, $\calT_H$ is interpreted as a mesh of the torus such that patches are continued over the periodic boundary \cite{OhlV17}.
For $S=T$ with $T\in \calT_H$ we call $U_m(T)$ the $m$-layer element patch and we refer to Figure~\ref{fig:patch} for a visualization.
By the quasi-uniformity of $\calT_H$ we further note that 
\begin{equation}
\max_{T\in \calT_H}\mathrm{card}\{K\in \calT_H\, |\,K\in U_m(T)\}\lesssim m^d.
\end{equation}

\begin{figure}
	\centering
	\begin{tikzpicture}[scale=4]
	\draw (0., 0.)--(0., 1.)--(1., 1.)--(1.,0.)--cycle;
	\filldraw[fill=blue!25!white](0./8., 1./8.) rectangle (4./8., 6./8.);
	\filldraw[fill=blue!25!white](7./8., 1./8.) rectangle (8./8., 6./8.);
	\filldraw[fill=blue!50!white](0./8., 2./8.) rectangle (3./8., 5./8.);
	\filldraw[fill=blue] (1./8., 3./8.) rectangle (2./8., 4./8.);
	
	\foreach \i in {1,...,8} { \draw (\i*1./8, 0)--(\i*1./8, 1); \draw(0,\i*1./8)--(1, \i*1./8);}
	\end{tikzpicture}
	\caption{A mesh element $T=U_0(T)$ (dark blue) and its patches $U_1(T)$ (intermediate blue) and $U_2(T)$ (light blue). The light blue squares on the very right belong to $U_2(T)$ because of the periodic continuation.}
	\label{fig:patch}
\end{figure}
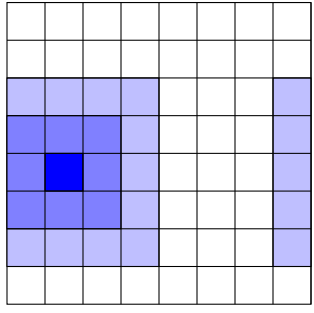

Fine-scale features are not captured in the coarse space $V_H$, i.e., a standard FEM on the coarse scale applied to~\eqref{eq:pbweak} does not yield faithful approximations.
We will characterize fine-scale parts of functions in $V$ as the kernel of a suitable interpolation operator.
We now introduce the required properties as well as an appropriate example.
Let $I_H: V\to V_H$ denote a bounded local linear projection operator, i.e., $I_H\circ I_H = I_H$, with the following stability and approximation properties for all $v\in V$
\begin{align}
\label{eq:IHapprox}
H^{-1}\,\|v-I_H v\|_{L^2(T)}+\|\nabla I_H v\|_{L^2(T)}&\lesssim \|\nabla v\|_{L^2(U_1(T))}.
\end{align}
A possible choice (which we use in our implementation of the method)
is to define $I_H:=E_H\circ\Pi_{H}$, where $\Pi_H: V\to \calQ_1(\calT_H)$ denotes the (local) $L^2$-projection.
$E_H$ is the averaging operator that maps discontinuous functions in $\calQ_1(\calT_H)$  to $V_H$ by
assigning to each free vertex the arithmetic mean of the corresponding
function values of the neighboring cells, that is, for any $v\in \calQ_1(\calT_H)$
and any vertex $z$ of $\calT_H$,
\begin{equation*}
(E_H(v))(z) =
\sum_{T\in\calT_H,\;z\in \overline{T}}v|_T (z) 
\bigg/
\operatorname{card}\{K\in\calT_H\,,\,z\in \overline{K}\}.
\end{equation*}
Note that again $\calT_H$ is understood as a mesh of the torus.
For further details on suitable interpolation operators we refer to \cite{EngHMP19}.

\subsection{Localized multiscale method}\label{subsec:LOD}
Denote by $\Vf=\ker I_H$ the kernel of $I_H$ which characterizes the functions with possible fine-scale variations in $V$.
Further, we introduce the following straightforward restriction of $\Vf$ to patches $U_m(S)$ via
\[\Vf(U_m(S)):=\{v\in \Vf\, |\, v|_{D\setminus U_m(S)}=0\}.\]
The LOD is based on (truncated) correction operators $\calC_m(A)$ defined by
\[\calC_m(A) v=\sum_{T\in\calT_H}\calC_{m,T}(A) v,\]
where the so-called element correction operator $\calC_{m,T}(A) :V\to \Vf(U_m(T))$ associated with the coefficient $A$ solves
\begin{equation}\label{eq:correclocal}
\bigl(A\nabla (\calC_{m,T}(A)v),  \nabla v^f\bigr)_{U_m(T)} = \bigl(A\nabla v,  \nabla v^f\bigr)_T\qquad \text{for all}\quad v^f\in \Vf(U_m(S)).
\end{equation}
Note that these local problems are well-posed by the Lax-Milgram lemma due to the uniform ellipticity of $A$.
The multiscale space $\Vms$ is now constructed as
\[\Vms:=V_H-\calC_m(A) V_H.\]
Denoting by $\calN$ the set of vertices of $\calT_H$ (understood as a mesh of the torus) and $\{\lambda_z\}_{z\in \calN}$ the nodal basis of $V_H$, $\{\lambda_z-\calC_m(A)\lambda_z\}_{z\in\calN}$ is a basis of $\Vms$.

The Petrov-Galerkin LOD (PG-LOD) for~\eqref{eq:pbweak} now reads as: Find $u_m^\mathrm{ms}\in \Vms$ such that
\begin{equation}\label{eq:pglod}
\fraka(u_m^\mathrm{ms}, v)=F(v)\qquad \text{for all}\quad v\in V_H,
\end{equation}
where we can write $u_m^\mathrm{ms}=u_m^H-\calC_m(A) u_m^H$ with $u_m^H\in V_H$.
In this Petrov-Galerkin variant only the ansatz functions are in the multiscale space, whereas the test functions are standard finite element functions.
The advantage over the Galerkin variant is that communication between different element correction operators is avoided. 
Hence, these correction operators, which are fine-scale quantities, do not need to be stored beyond the assembly of local stiffness matrix contributions.
We emphasize that in practical computations $u_m^H\in V_H$ is first determined by solving the linear system associated with~\eqref{eq:pglod}. If required, the element correction operators using $u_m^H$ can be computed to yield the full multiscale approximation $u_m^\mathrm{ms}$.
The PG-LOD~\eqref{eq:pglod} is well-posed for sufficiently large $m$ where no stability issues are reported in practice even for small choices $m=2,3$, cf.~\cite{ElfGH15,HelKM20}.
From \cite[Thm.~2]{ElfGH15} we obtain the following a priori error estimates
\begin{equation}\label{eq:pgloderror}
\|A^{1/2}\nabla (u-u_m^\mathrm{ms})\|_{L^2(D)}+ \|u-u_m^H\|_{L^2(D)}\lesssim (H+m^{d/2}\gamma^m)\, \|f\|_{L^2(D)}
\end{equation}
for some $0<\gamma<1$ independent of $H$ and $m$.
Choosing $m\gtrsim |\log H|$, these estimates essentially show that (i) $u_m^\mathrm{ms}$ converges linearly to $u$ in the energy norm and (ii) $u_m^H$ converges linearly to $u$ in the $L^2(D)$-norm.
In other words, $u_m^H$ is a good $L^2$-approximation to $u$, while the correction operators and thus $u_m^\mathrm{ms}$ are necessary to obtain a good $H^1$-approximation of $u$.

\subsection{Offline-online strategy}\label{subsec:LODnew}
In this section, we suggest an offline-online strategy for the fast computation of the left-hand side in~\eqref{eq:pglod} for many different realizations $A$.
For $v,w\in V_H$, we denote $\frakb(v, w):=\fraka(v-\calC_m(A) v, w)$ and observe that
\begin{equation}\label{eq:bilin}
\frakb(v,w)=\sum_{T\in \calT_H}\frakb_T(v, w)
\end{equation}
with
\begin{equation}\label{eq:bilinlocal}
\frakb_T(v,w):=\int_{U_m(T)}A(x)(\chi_T \nabla v -\nabla (\calC_{m,T}(A) v))(x)\cdot \nabla w(x)\, dx,
\end{equation}
where $\chi$ denotes the characteristic function.
We will from now on assume that the mesh size $H$ is an integer multiple of the periodicity length $\varepsilon$.
This implies that $\frakb_T(\cdot, \cdot)$ for the coefficient $A_\varepsilon$ is identical for every mesh element $T\in \calT_H$.
Hence, only $\frakb_T(\cdot, \cdot)$ for a single $T\in \calT_H$ is required in order to assemble $\frakb(\cdot, \cdot)$ associated with $A_\varepsilon$.

\smallskip
\noindent
\textbf{Offline phase.\hspace{1ex}} 
Fix an element $T\in \calT_H$. Let $J:=\{k\in \mathbb{Z}^d\, |\, \varepsilon(k+Q)\subset U_m(T)\}$ be the index set of possible defects in the patch $U_m(T)$ and denote by $N:= \mathrm{card} J$ its cardinality.
\begin{figure}
	\begin{subfigure}{0.45\textwidth}
		\includegraphics[width=\textwidth, trim=18mm 7mm 18mm 12mm, clip=true]{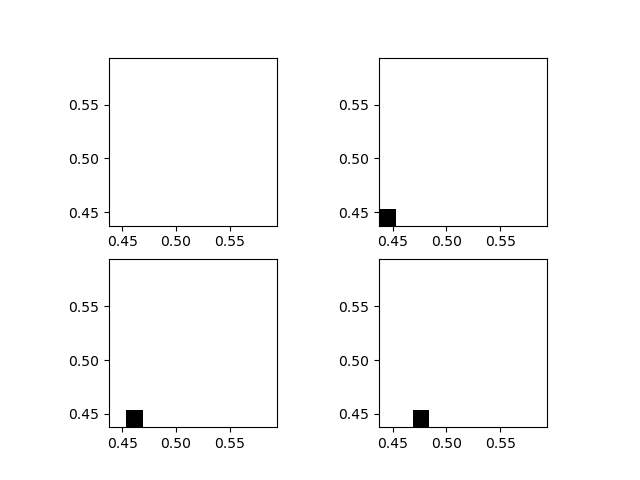}
		\caption{Random checkerboard}
	\end{subfigure}%
	\hspace{5ex}%
	\begin{subfigure}{0.45\textwidth}
		\includegraphics[width=\textwidth, trim=18mm 7mm 18mm 12mm, clip=true]{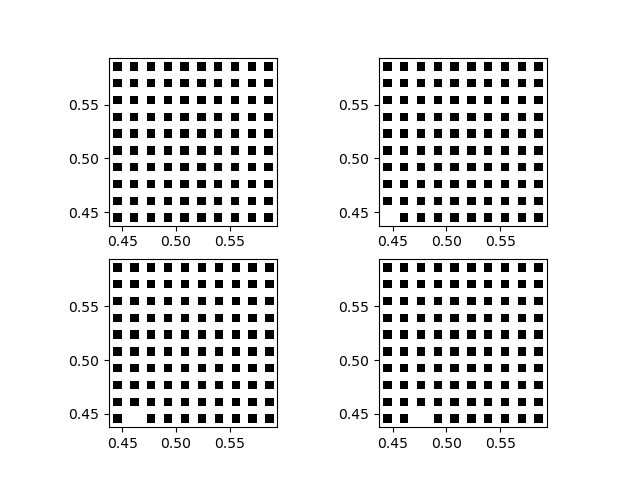}
		\caption{Random defect inclusions}
	\end{subfigure}
	\caption{Offline coefficients $A_0, A_1, A_2, A_3$ (from top left to bottom right) for random checkerboard of Example~\ref{ex:randcheck} (left) and random defect inclusions of Example~\ref{ex:randdef} (right). Generated with $\alpha=0.1$ (white), $\beta=1$ (black), $\varepsilon=2^{-6}$, $H=2^{-5}$, and $m=2$.}
	\label{fig:offline}
\end{figure}

We introduce a bijective mapping $\sigma: \{1,\ldots, N\}\to J$.
Further, we set 
\begin{equation}\label{eq:offlinecoeff}
A_i:=\begin{cases}
A_\varepsilon|_{U_m(T)}, \qquad\qquad \qquad\qquad\quad i=0,\\
A_\varepsilon|_{U_m(T)}+\chi_{\varepsilon(\sigma(i)+Q)}B_\varepsilon $, \qquad $i=1, \ldots, N
\end{cases}
\end{equation}
as our stored offline ``basis'' of coefficients.
Intuitively, this means that $A_i$ is constructed from $A_\varepsilon$ by introducing a single defect.
For the two examples~\ref{ex:randcheck} and~\ref{ex:randdef} of random coefficients of Figure~\ref{fig:randcoeffs} in the introduction, some corresponding $A_i$ are depicted in Figure~\ref{fig:offline} left and right, respectively.

In the offline phase, we compute the local LOD stiffness matrix contributions 
\begin{equation}
\frakb_T^i(\lambda_j, \lambda_k)=\int_{U_m(T)}A_i(x)\bigl(\chi_T \nabla \lambda_j -\nabla (\calC_{m,T}(A_i)\lambda_j)\bigr)(x)\cdot \nabla \lambda_k(x)\, dx,
\end{equation} 
where $\{\lambda_j\}$ is the set of finite element basis functions spanning $V_H$ and $\calC_{m,T}(A_i)$ denotes the element correction operator associated with the coefficient $A_i$.
Note that the stiffness matrix contribution for the fixed element $T$ itself is a coarse-scale object and inexpensive to store.
Additionally, we also assemble the load vector, i.e., the right-hand side of~\eqref{eq:pglod}. For this assembly, we have to consider all mesh elements, but the load vector is the same for all coefficients.

\smallskip\noindent
\textbf{Online phase.\hspace{1ex}} 
Given a sample coefficient $A$ of the form~\eqref{eq:weaklyrandom} and~\eqref{eq:randomness}, there are $\mu_i\in \rz$, $i=0, \ldots, N$ such that $\sum_{i=0}^N \mu_i=1$ and 
\begin{equation}\label{eq:decompA}
A|_{U_m(T)}=\sum_{i=0}^N \mu_i A_i
\end{equation}
for any $T\in \calT_H$.
More specifically, $\mu_i$ for $i=1, \ldots, N$ is determined from the value of $\hat b_p^j$ for a certain $j$. In particular, we have $\mu_i\in \{0,1\}$ for $i=1,\ldots, N$ and
$\mu_0=1-N_{\mathrm{def}}$ where $N_{\mathrm{def}}$ denotes the number of defects in the patch $U_m(T)$.
Note that $\mu_i$ depends (implicitly) on $T$ and that its calculation is cheap.
We further emphasize that \eqref{eq:decompA} is no assumption on $A$, but actually holds for every $A$ of the form \eqref{eq:weaklyrandom} and \eqref{eq:randomness} due to the definition of the $A_i$ in \eqref{eq:offlinecoeff}.

In the online phase, we calculate the global LOD stiffness matrix as a combination of the offline quantities as follows.
With the $\mu_i$ at hand, we compute the local \emph{combined} LOD stiffness matrix contributions as
\begin{equation}\label{eq:onlinebt}
\tilde \frakb_T(\lambda_j,\lambda_k)=\sum_{i=0}^N \mu_i \frakb_T^i(\lambda_j,\lambda_k).
\end{equation}
The global combined bilinear form $\tilde{\frakb}$ is defined as usual via $\tilde{\frakb} =\sum_{T\in \calT_H}\tilde{\frakb}_T.$
Effectively, the sum in~\eqref{eq:onlinebt} only contains $N_{\mathrm{def}}+1$ nonzero terms.
Roughly, only a fraction of $p$ terms thus needs to be considered each time.
After the assembly of the global stiffness matrix, we compute $\tilde u_m^H\in V_H$ as the solution of
\begin{equation}\label{eq:offlineonline}
\tilde{\frakb}(\tilde u_m^H, v_H)=F(v_H)\qquad \text{for all}\quad v_H\in V_H.
\end{equation}
Note that the underlying linear system is of small dimension.

\eqref{eq:pgloderror} shows that $u_m^H$ is a good $L^2$ approximation to the exact solution $u$, but we have to add correctors and consider $u_m^{\mathrm{ms}}$ to obtain a good approximation in $H^1$.
	Hence, also $\tilde u_m^H$ is only expected to be a good $L^2$ approximation to $u$. We define the upscaled solution for our offline-online strategy as 
	\begin{equation}\label{eq:offlineonlineupscaled}
	\tilde u_m^{\mathrm{ms}}=\tilde u_m^H-\tilde \calC_m \tilde u_m^H,\qquad  \text{where} \quad \tilde \calC_m:=\sum_{T\in \calT_H}\sum_{i=0}^N \mu_i \calC_{m,T}(A_i).
	\end{equation}
	This requires to store the correctors $\calC_{m,T}(A_i)$ in the offline phase, but as this is only done for one single element $T$, it is affordable.
	With the offline correctors available, also $\tilde{\calC}_m$ can be assembled quickly and $\tilde u_m^{\mathrm{ms}}$ is readily available.

\smallskip
More details on the (efficient) implementation of the offline-online strategy are given in Section~\ref{subsec:algo}.
As already indicated by the notation, $\tilde{\frakb}$ and $\tilde{u}_m^H$ are only approximations to the true PG-LOD form $\frakb$ and solution $u_m^H$ associated with the sample coefficient $A$, respectively.
We analyze the error committed by the new strategy in the next section.

\begin{remark}\label{rem:bdrycondoffline}
	Because of the periodic boundary conditions and the box-type domain, the local LOD stiffness matrix for $A_\varepsilon$ and the choice of offline coefficients are identical for every element $T$.
	In case of Neumann or Dirichlet boundary conditions or more general domains, there are several representative configurations for the possible patches. 
	One can adapt the offline phase to these situations by calculating and storing the matrix contributions for all possible patch configurations and the associated offline coefficients.
	If the number of possible patch configurations is small, e.g., for a highly structured mesh and domain, the additional effort required may still be feasible.
\end{remark}

\begin{remark}\label{rem:structoffline}
	The assumption of periodic $A_\varepsilon, B_\varepsilon$ with $H$ and the length of the domain being integer multiples of $\varepsilon$ is important to guarantee that  the offline and the samples coefficients all have the same structure for each mesh element. 
	Otherwise, we would need to perform the above described offline phase for all mesh elements. This, in turn, increases the computational time and the storage costs.
	While the additional costs in run-time might be compensated by an effective online phase if sufficiently many samples are considered, the storage may become a bottleneck.
	In future work, one might therefore try to reduce the number of offline coefficients per mesh element or consider adaptive online strategies, cf.~Remark~\ref{rem:indicator}.
\end{remark}

\begin{remark}
	In this presentation we focus on the case when $A = \sum_{i=0}^N \mu_i A_i$. If $A$ is not exactly a linear combination of the $A_i$'s one would instead need to find optimal weights $\{\mu_i\}_{i=0}^N$ to minimize the error. In Section~\ref{sec:analysis}, we present an error estimator presented that bounds the error also for this case. Exactly how to do this optimization is outside the scope of this presentation and we leave it for future investigation.
\end{remark}

	\subsection{Connection to homogenization approaches}\label{subsec:hom}
	Since we assume the diffusion coefficient to be $\varepsilon$-periodic with randomly distributed defects, a natural question to ask is how the proposed strategy relates to  homogenization approaches.
	Under the assumptions of stationarity and ergodicity, which are satisfied in our setting, (quantitative) stochastic homogenization characterizes asymptotic expansions for the expectation of the solution and (homogenized) limit equations for the terms in these expansions. We emphasize that these results are asymptotic in the sense that they consider the limit $\varepsilon\to 0$.
	We refer to \cite{ArmKuuMou19,GloNeuOtt19} for overviews.
	
	Based upon these results, a number of different approaches have emerged that aim to compute approximations to the stochastic homogenization matrix, which then  allows a cheap calculation of the homogenized solution as an approximation to $u_\varepsilon$.
	Giving a complete overview of these approaches is far beyond the scope of this paper, but we would like to highlight the contributions \cite{AnaLb11,AnaLb12} as well as \cite{Lm15} that consider the same setting of a periodic diffusion coefficient perturbed by rare random defects.
	In \cite{AnaLb11,AnaLb12}, the authors deduce an expansion of the stochastic homogenization matrix in terms of the defect probability. 
	The work \cite{Lm15} additionally proposes a control variate technique to reduce the variance.
	The zeroth order term in the expansion is the homogenized matrix for $A_\varepsilon$ obtained by classical periodic homogenization. The first and second order terms are homogenization averages involving cell problems solutions where the coefficient has one or two defects, respectively.
	
	Note that the choice of our offline coefficients is somewhat related since $A_0=A_{\mathrm{per}}$ and the $A_i$ introduce a single defect.
	Further, the LOD is connected to homogenization theory in the way that, for a fixed sample coefficient $A$, it can be reformulated as a finite element-type discretization using a quasi-local integral kernel operator, which can even be approximated by a local ``homogenized'' coefficient in certain situations, see \cite{GalP17} for details. 
	By taking the expectation, a deterministic quasi-local or local coefficient can be defined and can be exploited to approximate the expected value of the solution numerically \cite{FisGP19,GalP19}.
	In this work, however, we are not interested in the stochastic homogenization matrix,  but the solution itself. 
	Further, we do not restrict ourselves to the expectation of the solution either, but through the efficient approximation of $u_\varepsilon$ with our suggested approach, any desired statistical information can be calculated through sampling. In that sense, the aims of our approach differ from homogenization approaches.
	Finally, we emphasize that we never assume $\varepsilon$ to be small (so that we would be in some homogenization regime). Instead, the LOD and our approach can be seen as a numerical homogenization technique on the scale $H$ in the sense that we provide a reasonable approximation to $u_\varepsilon$ that is defined on the coarse mesh $\calT_H$.

\section{A priori eror analysis}\label{sec:analysis}
In this section, we discuss the well-posedness of~\eqref{eq:offlineonline} as well as error estimates for $u-\tilde{u}_m^H$.
To accomplish this, we start by studying the consistency error  $\frakb-\tilde{\frakb}$.
We first consider the one-dimensional case where the correction operators can be explicitly computed and then discuss the generalization to several dimensions.
As in the previous sections, $A$ denotes the true coefficient with associated PG-LOD bilinear form $\frakb$ and the bilinear form of the offline-strategy $\tilde{\frakb}$ is defined via~\eqref{eq:onlinebt}. 

In the one-dimensional setting with $I_H$ chosen as the nodal interpolation operator, the corrector problems automatically localize to single coarse elements, i.e., $m=0$ is sufficient.
Moreover, the correction operators can be explicitly calculated.
Hence, \cite{HenMPSVK20} provides the following result on $\frakb_T$ from~\eqref{eq:bilinlocal} in $d=1$: It holds for any $v,w\in V_H$ that
\[\frakb_T(v, w) = (A_{\mathrm{harm}}|_T \nabla v, \nabla w)_T\]
with the element-wise constant coefficient $A_{\mathrm{harm}}$ defined as
\begin{equation*}
A_{\mathrm{harm}}|_T:=\Bigl(\frac{1}{|T|}\int_T A^{-1}\, dx\Bigr)^{-1}.
\end{equation*}
This means that $\frakb$ can be written as a finite element bilinear form with a modified coefficient, namely the element-wise harmonic mean.
Similarly, $\tilde{\frakb}$ can be written as finite element bilinear form associated with
\begin{equation}\label{eq:Aharmmu}
A_\mathrm{harm}^\mu|_T:=\sum_{i=0}^N \mu_i A_\mathrm{harm}^i|_T,
\end{equation}
where $A_\mathrm{harm}^i$ denotes the harmonic mean of $A_i$.
Let $N_\mathrm{def}$ denote the number of defects in $T$ for the given realization. In the following, we will write $N_\mathrm{def}=\theta_\mathrm{def, T}N$ with $N$ the number of possible defect locations (in $T$), cf.~Section~\ref{subsec:LODnew}.
Note that in this one-dimensional setting, we have $N=H/\varepsilon$.
We abbreviate $\theta_{\mathrm{def}} = \max_{T\in \calT_H} \theta_{\mathrm{def}, T}$.
The representation of $\tilde{\frakb}$ in the one-dimensional setting allows for the following a priori bound.

\begin{theorem}\label{thm:error1D}
	If $D$ is one-dimensional and $I_H$ is the nodal interpolation operator, the consistency error between $\frakb$ from~\eqref{eq:bilin} and $\tilde{\frakb}$ defined via~\eqref{eq:onlinebt} fulfills for any $v, w\in V_H$
	\begin{equation}\label{eq:error1D}
	|(\frakb-\tilde{\frakb})(v,w)|\leq \frac{\beta}{\alpha}\Bigl(\frac{\beta-\alpha}{\alpha}\Bigr)^2 |Q|^2\Bigl(\frac{\varepsilon}{H} \theta_{\mathrm{def}}+2\theta_\mathrm{def}^2\Bigr)\, \|v\|_A \|w\|_A.
	\end{equation}
\end{theorem}
This theorem clearly underlines why the approach works for small defect probabilities and hence, for small $\theta_\mathrm{def}$.
We emphasize that the $O(\theta_\mathrm{def})$ error contribution is multiplied by  the small factor $\varepsilon/H<1$.
The proof of Theorem~\ref{thm:error1D} is presented in Appendix~\ref{sec:appendix}.

Extending estimate~\eqref{eq:error1D} to higher dimensions is challenging because we do not have an explicit and local representation of $\frakb$.
In the following, we present an upper bound on the consistency error $\frakb-\tilde \frakb$ that is computable in an a posteriori manner.
We emphasize that the result does not require $A|_{U_m(T)} = \sum_{i=0}^N \mu_i A_i$.
We abbreviate $\bar{A}=\sum_{i=0}^N \mu_i A_i$. 

\begin{theorem}\label{thm:errorhigherdim}
	Define for any $T\in \calT_H$
	\begin{equation}\label{eq:indicmultiple}
	\begin{aligned}
	E_{T}^2&:=\max_{v \in V_H:v|_T}\frac{\|(A^{1/2}-A^{-1/2}\bar{A})\chi_T\nabla v- \sum_{i=0}^N \mu_i(A^{1/2}-A^{-1/2}A_i)\nabla (\calC_{m,T}(A_i) v)\|^2_{L^2(U_m(T))}}{\|v\|_{A, T}^2}.
	\end{aligned}
	\end{equation}
	Then, for any $v,w\in V_H$ it holds that
	\begin{equation}\label{eq:consistencymultiple}
	|(\tilde{\frakb}-\frakb)(v,w)|\lesssim m^{d/2}\,  \Bigl(\max_{T\in \calT_H} E_T\Bigr)\, \|v\|_{A}\, \|w\|_{A}.
	\end{equation}
\end{theorem}
In practice, $E_T$ is computed as the largest eigenvalue of an eigenvector problem of dimension $2^{d}$, which is the dimension of the coarse scale finite element space on one element $T$. We only need to store $A_i$ and $\calC_{m,T}(A_i)$ on a single patch $U_m(T)$ and have access to the sampled coefficient $A$. We refer to Section~\ref{subsec:indicator} for details on these implementation aspects.
Note that for $N=0$, i.e., a single ``reference'' coefficient $A_i$, $E_T$ coincides with $e_{u,T}$ defined in \cite[Lemma 3.3]{HelM19}.
Theorem~\ref{thm:errorhigherdim} can thus be interpreted as a generalization to the case of several reference coefficients.

\begin{remark}\label{rem:indicator}
	In the present contribution, we see $E_T$ as a computational tool to easily obtain an upper bound on the actual error without the need to compute $u_m^H$ and, in particular, the correction operators $\calC_{m,T}(A)$ associated with $A$.
	Note that the computation of $E_T$ is not necessary in the offline-online strategy if no error control is required.
	Since $E_T$ can be evaluated without $\calC_{m,T}(A)$, we will investigate its use to build up the offline coefficients $A_i$ or to enrich them during the online phase in future work.
	For instance, in a similar spirit as in \cite{HelM19,HelKM20}, one could use $E_T$ to mark elements where the corrector $\mathcal C_{m,T}(A)$ needs to be newly computed.
		Although $E_T$ gives only an upper bound of the local error between $\calC_{m,T}(A)$ and $\widetilde \calC_{m,T}$ and no lower bound, comparing values of $E_T$ over all elements gives a good indication where such a new computation of $\calC_{m,T}(A)$ will be most beneficial.
		Further, our numerical experiment in Section~\ref{subsec:numexp:ET} indicates that, indeed, $E_T$ is a good indicator for the local error.
\end{remark}

\begin{proof}[Proof of Theorem~\ref{thm:errorhigherdim}]
	Let $v, w\in V_H$ be arbitrary but fixed.
	We will show that for any $T\in \calT_H$ it holds that
	\begin{equation}\label{eq:consislocal}
	|(\frakb_T-\tilde{\frakb}_T)(v,w)|\lesssim E_T \| v\|_{A, T}\, \|w\|_{A, U_m(T)}.
	\end{equation}
	Let us first illustrate how this implies the assertion of the theorem:
	\begin{align*}
	|(\frakb-\tilde{\frakb})(v,w)|&\leq \sum_{T\in \calT_H}|(\frakb_T-\tilde{\frakb}_T)(v,w)|\\
	&\lesssim \sum_{T\in \calT_H}E_T \|v\|_{A, T}\, \|w\|_{A, U_m(T)}\\*
	&\lesssim m^{d/2} \Bigl(\max_{T\in\calT_H}E_T\Bigr)\,  \|v\|_{A}\, \|w\|_{A}.
	\end{align*}
	
	Let us now prove~\eqref{eq:consislocal}. We abbreviate $\calC_{m,T}^i=\calC_{m,T}(A_i)$ and $\calC_{m,T} = \calC_{m,T}(A)$.
	We have
	\begin{align*}
	&\!\!\!\!\frakb_T(v,w)-\tilde \frakb_T(v,w)\\*
	&=\bigl(A(\chi_T\nabla-\nabla \calC_{m,T})v, \nabla w\bigr)_{U_m(T)}- \sum_{i=0}^N\mu_i\bigl(A_i(\chi_T\nabla -\nabla \calC_{m,T}^i )v, \nabla w\bigr)_{U_m(T)}\\
	&=\Bigl( (A-\bar{A})\chi_T \nabla v-\sum_{i=0}^N \mu_i (A-A_i)\nabla \calC_{m,T}^i v, \nabla w\Bigr)_{U_m(T)}\\*
	&\qquad -\Bigl(A\nabla \Bigl(\calC_{m,T}-\sum_{i=0}^N \mu_i \calC_{m,T}^i\Bigr)v, \nabla w\Bigr)_{U_m(T)}\\
	&\leq E_T \|v\|_{A, T}\, \|w\|_{A, U_m(T)} + \Bigl\|\Bigl(\calC_{m,T}-\sum_{i=0}^N \mu_i \calC_{m,T}^i\Bigr)v\Bigr\|_{A, U_m(T)} \, \|w\|_{A, U_m(T)}.
	\end{align*}
	It remains to estimate the second term.
	We abbreviate $z=\calC_{m, T} v-\sum_{i=0}^N \mu_i \calC_{m,T}^i v\in \Vf(U_m(T))$.
	We deduce by the definition of $\calC_{m,T}^i v$ that
	\begin{equation}\label{eq:correcerror}
	\begin{aligned}
	\|z\|_{A, U_m(T)}^2&=\Bigl(A \nabla\Bigl(\calC_{m,T} v-\sum_{i=0}^N \mu_i \calC_{m,T}^i v\Bigr),\nabla z\Bigr)_{U_m(T)}\\
	&=(A\nabla v,\nabla z)_T-\Bigl(A \sum_{i=0}^N \mu_i\nabla(\calC_{m, T}^i v),\nabla z\Bigr)_{U_m(T)}\\
	&=((A-\bar{A})\nabla v,\nabla z)_T-\Bigl(\sum_{i=0}^N (A-A_i)\mu_i \nabla \calC_{m,T}^i v,\nabla z\Bigr)_{U_m(T)}\\
	&\leq \Bigl\| (A^{1/2}-A^{-1/2}\bar{A})\chi_T\nabla v-\sum_{i=0}^N \mu_i (A^{1/2}-A^{-1/2}A_i)\nabla \calC_{m,T}^i v\Bigr\|_{L^2(U_m(T))}\,  \|z\|_{A, U_m(T)}\\
	&\leq E_T\| v\|_{A, T}\, \|z\|_{A, U_m(T)},
	\end{aligned}
	\end{equation}
	which finishes the proof.
\end{proof}

Theorems~\ref{thm:error1D} and~\ref{thm:errorhigherdim} provide bounds on the consistency error
\begin{equation*}
\eta := \sup_{v\in V_H\setminus \{0\}}\sup_{w\in V_H\setminus\{0\}}\frac{|(\frakb-\tilde{\frakb})(v,w)|}{\| v\|_{A}\, \|w\|_{A}}.
\end{equation*}
If the consistency error is sufficiently small, well-posedness of~\eqref{eq:offlineonline} is guaranteed and we also obtain an error estimate as detailed in the next corollary.
The additional term $m^{d/2}\bigl(\max_{T\in \calT_H}E_T\bigr)$ in the $H^1$ bound is due to the additional corrector error $\calC_m-\tilde \calC_m$.

\begin{corollary}\label{cor:error}
	There exist $m_0>0$ and $\eta_0>0$ such that, if $m>m_0$ and $\eta<\eta_0$, \eqref{eq:offlineonline} is well-posed and, further, the error between the solution $u$ of~\eqref{eq:pbweak} and the solution $\tilde{u}_m^H\in V_H$ of~\eqref{eq:offlineonline} and its upscaled version \eqref{eq:offlineonlineupscaled} satisfies
	\begin{align*}
	\|u-\tilde{u}_m^H\|_{L^2(D)}&\lesssim (H+m^{d/2}\gamma^m+ \eta)\|f\|_{L^2(D)},\\
	\|u-\tilde u_m^{\mathrm{ms}}\|_{A}&\lesssim \Bigl(H+m^{d/2}\gamma^m+m^{d/2}\Bigl(\max_{T\in \calT_H}E_T\Bigr)\Bigr)\|f\|_{L^2(D)},
	\end{align*}
	where $\gamma$ is the constant for the exponential decay from \eqref{eq:pgloderror}.
\end{corollary}

\begin{proof}
	We proceed similar to the proof of Theorem~4.1 in \cite{HelKM20}. 
	To show the well-posedness of~\eqref{eq:offlineonline}, we prove the coercivity of $\tilde{\frakb}$ if $\eta<\eta_0$ and $m>m_0$.
	We again abbreviate $\calC_m=\calC_m(A)$ and $\calC_{m,T}=\calC_{m,T}(A)$. 
	By $\calC_{\infty, T}$ we denote the element correction operator $\calC_{m,T}$ with $m=\infty$, i.e., where the integral on the left-hand side of~\eqref{eq:correclocal} is taken over the whole domain $D$.
	We set $\calC_\infty=\sum_{T\in\calT_H}\calC_{\infty, T}$ and note that $\fraka(v-\calC_\infty v, \calC_\infty w)=0$ for all $v,w\in V_H$.
	Let $v\in V_H$ be arbitrary. We deduce
	\begin{align*}
	\tilde{\frakb}(v,v)&=\frakb(v,v)+(\tilde{\frakb}-\frakb)(v,v) =\fraka(v-\calC_m v, v)+(\tilde{\frakb}-\frakb)(v,v)\\
	&=\fraka(v-\calC_\infty v, v)+\fraka(\calC_\infty v-\calC_m v, v) +(\tilde{\frakb}-\frakb)(v,v)\\
	&\geq (c_1+c_2m^{d/2}\gamma^m -\eta)\|v\|^2_{A},
	\end{align*}
	where we used
	\[\fraka(v-\calC_\infty v, v)=\fraka(v-\calC_\infty v, v-\calC_\infty v)=\|v-\calC_\infty v\|^2_{A}\geq c_1 \|I_H(v-\calC_\infty v)\|^2_{A}=\|v\|_{A}^2\]
	and 
	\[\|(\calC_\infty-\calC_m)v\|_{A}\leq c_2 m^{d/2}\gamma^m \|A^{1/2}\nabla v\|_{L^2(D)}\]
	in the last step.
	Clearly, the are $m_0$ and $\eta_0$ such that for $m>m_0$ and $\eta<\eta_0$, $c_1-c_2m^{d/2}\theta^m -\eta$ can be bounded from below by a positive constant. This shows the coercivity of $\tilde{\frakb}$.
	
	For the error estimates, we use the triangle inequality to get
	\[\|u-\tilde{u}_m^H\|_{L^2(D)}\leq \|u-u_m^H\|_{L^2(D)}+\|u_m^H-\tilde{u}_m^H\|_{L^2(D)},\]
	where the first term is estimated in~\eqref{eq:pgloderror}.
	By the coercivity of $\tilde{\frakb}$ we obtain for $u_m^H-\tilde{u}_m^H$ that
	\begin{align*}
	\|u_m^H-\tilde{u}_m^H\|_{A}^2 &\lesssim \tilde{\frakb}(u_m^H-\tilde{u}_m^H, u_m^H-\tilde{u}_m^H)= \tilde{\frakb}(u_m^H, u_m^H-\tilde{u}_m^H)-F( u_m^H-\tilde{u}_m^H) \\
	&= (\tilde{\frakb}-\frakb)(u_m^H, u_m^H-\tilde{u}_m^H)\\
	&\lesssim \eta\, \|f\|_{L^2(D)}\,  \|u_m^H-\tilde{u}_m^H\|_{A},
	\end{align*}
	where we used the stability of the PG-LOD solution $u_m^H$ in the last step. Application of Friedrich's inequality yields the $L^2$-bound.
	
	For the $H^1$ bound, we first note that for any $v\in V_H$
	\begin{align*}
	\|(\calC_m-\tilde\calC_m)v_H\|_A^2&= \Bigl\|\sum_{T\in \calT_H}(\calC_{m,T}-\tilde\calC_{m,T})v_H\|_A^2\\
	&\lesssim \sum_{T\in \calT_H}m^d\|(\calC_{m,T}-\tilde\calC_{m,T})v_H\|_A^2\\
	&\lesssim m^d\Bigl(\max_{T\in \calT_H}E_T\Bigr)\|v\|_A^2,
	\end{align*}
	where we used the finite overlap and estimate \eqref{eq:correcerror} from the proof of Theorem~\ref{thm:errorhigherdim}.
	We now deduce that
	\begin{align*}
	\|u-\tilde u_m^{\mathrm{ms}}\|_A&\leq \|u-u_m^{\mathrm{ms}}\|_A+\|u_m^H-\calC_m u_m^H-\tilde u_m^H+\tilde \calC_m\tilde u_m^H\|_A\\
	&\leq \|u-u_m^{\mathrm{ms}}\|_A+\|u_m^H-\tilde u_m^H\|_A+\|\tilde\calC_m(u_m^H-\tilde u_m^H)\|_A+\|(\tilde \calC_m-\calC_m)u_m^H\|_A\\
	&\lesssim \|u-u_m^{\mathrm{ms}}\|_A+\|u_m^H-\tilde u_m^H\|_A+\|(\tilde \calC_m-\calC_m)u_m^H\|_A,
	\end{align*}
	where we used in the last step that
	\[\|\tilde \calC_m v\|_A\leq \|\calC_\infty v\|_A+\|(\calC_\infty-\calC_m)v\|_A+\|(\calC_m-\tilde \calC_m)v\|_A\lesssim \Bigl(1+m^{d/2}\gamma^m+m^{d/2}\Bigl( \max_{T\in \calT_H}E_T\Bigr)\Bigr)\|v\|_A.\]
	Combining the above estimates with \eqref{eq:pgloderror}, the already obtained estimate for $\|u_m^H-\tilde u_m^H\|_A$ in the proof of the $L^2$ bound, and the stability of the PG-LOD solution finishes the proof.
\end{proof}

\begin{remark}
	In the one-dimensional case, $\tilde{\frakb}$ is coercive if and only if $A_\mathrm{harm}^\mu$ in~\eqref{eq:Aharmmu} is positive on each element.
	A sufficient condition -- alternative to a small consistency error -- is $B_\varepsilon\geq 0$, i.e., the random perturbation is always additive to $A_\varepsilon$. 
	In more detail, if $B_\varepsilon\geq 0$, we can show $A_\mathrm{harm}^i|_T\geq A_\mathrm{harm}^0|_T$ for $i=1,\ldots, N$. Because of $\mu_i\geq 0$ for $i=1,\ldots, N$ and $\sum_{i=0}^N \mu_i=1$, we deduce $A_\mathrm{harm}^\mu|_T\geq \alpha>0$.
	The sufficient condition $B_\varepsilon\geq 0$ is satisfied for Example~\ref{ex:randcheck}, but not for Example~\ref{ex:randdef}.
\end{remark}

\section{Implementation aspects}\label{sec:impl}
This section deals with implementation details specific to the presented offline-online strategy with a focus on the algorithm  (Section~\ref{subsec:algo}), its comparison concerning run-time to  the method of \cite{HelKM20} (Section~\ref{subsec:complexity}), and the implementation of $E_T$ from Theorem~\ref{thm:errorhigherdim} (Section~\ref{subsec:indicator}).
For the general implementation of the (PG-)LOD we refer to \cite{EngHMP19} and \cite[Ch.~7]{MalP20}.

\subsection{Algorithm for the offline-online strategy and memory consumption}\label{subsec:algo}
Algorithm~\ref{algo:offlineonline} shows how to carry out the procedure from Section~\ref{subsec:LODnew} computationally.
For simplicity, we focus on the computation of $\tilde u_m^H$ only and omit the additions for the upscaled version $\tilde u_m^{\mathrm{ms}}$.
\begin{algorithm}
	\caption{Offline-online strategy}
	\label{algo:offlineonline}
	\begin{algorithmic}[1] % The number tells where the line numbering should start
		\State \textbf{input:} Problem data $A_\varepsilon$, $B_\varepsilon$, $Q$, $f$ (cf.~Section~\ref{sec:setting})
		\State Pick $m$
		\State Fix $T\in \calT_H$ \Comment{start offline phase}
		\State Precompute and save offline coefficients $\{A_0, A_1, \ldots A_N\}$ according to~\eqref{eq:offlinecoeff}
		\For {$i=0, \ldots N$} 
		\State Precompute $\calC_{m,T}(A_i)\lambda_j$ for all $j$  (discard at end of iteration)
		\State Precompute and save $\frakb_T^i(\lambda_j, \lambda_k)$ for all $j$ and $k$
		\EndFor
		\State Precompute and save $F(\lambda_j)$ for all $j$ \Comment{end offline phase}
		\ForAll {sample coefficients $A$} \Comment{start online phase}
		\ForAll {$T\in \calT_H$}
		\State Compute $\mu_i$ such that $A|_{U_m(T)}=\sum_{i=0}^N \mu_i A_i$
		\State Compute and save $\tilde \frakb_T(\lambda_j, \lambda_k)=\sum_{i=0}^N \mu_i \frakb_T^i(\lambda_j, \lambda_k)$ for all $j$ and $k$
		\EndFor
		\State Assemble stiffness matrix $\tilde{K}_{kj}:=\sum_{T\in \calT_H}\tilde{\frakb}_T(\lambda_j, \lambda_k)$ 
		\State Solve for $\tilde{u}_m^H$ according to~\eqref{eq:offlineonline} 
		\EndFor
	\end{algorithmic}
\end{algorithm} %\phantom{.}\\[1.5ex]

It starts with setting up the offline coefficients.
Due to the periodicity and the weakly random structure, only the fine-scale representations of $A_\varepsilon$ and $B_\varepsilon$ on a single coarse element $T$ need to be stored with a cost of order $(H/h)^d$.
$\frakb_T^i$ is computed for all offline coefficients, but only for a single coarse element. The storage cost is of order $N\, m^d$ where the number of offline coefficients $N$ is of the order $(mH/\varepsilon)^d$.
We especially emphasize that, for moderate $m$ and not too coarse $H$, we have $N<O(\varepsilon^{-d})= \mathrm{card} I$ with the index set $I$ from~\eqref{eq:randomness}.
This means that there are less possible defects in the element patch $U_m(T)$ than in $D$.
The pre-computations for the offline coefficients can be executed in parallel.
Finally, we also assemble and store the load vector in the offline phase with a cost of order $H^{-d}$.

In the online phase, we perform a loop over the sample coefficients (i.e., Monte-Carlo-type sampling) which again can be executed in parallel.
For each coefficient, there is a loop over the elements which is also parallelizable. For each element, we extract the $\mu_i$ forming the representation of $A$ in terms of the offline coefficients. Due to the representation~\eqref{eq:randomness}, this does not require extensive computations.
We emphasize that many $\mu_i$ are identically zero for rare perturbations so that the sum for $\tilde{\frakb}_T$ is evaluated cheaply.
Finally, the coarse-scale linear system with $\tilde{\frakb}$ is assembled and solved. The assembly of the stiffness matrix involves a reduction over $T$, but only coarse-scale quantities like $\tilde{\frakb}_T$ of amount $m^dH^{-d}$ are needed between different elements, cf.~\cite{HelKM20}.
From the coarse-scale solution $\tilde{u}_m^H$ for each sample coefficient, we can of course compute quantities of interest or agglomerate statistical information.

\subsection{Run-time complexity}\label{subsec:complexity}
Based on Algorithm~\ref{algo:offlineonline}, let us briefly comment on the run-time complexity in comparison to the standard LOD and the LOD with local updates according to \cite{HelKM20}.
We consider the time taken for the stiffness matrix assembly for $M_\mathrm{samp}$ sample coefficients and consider completely sequential versions of all methods.
In the following, $t_{\mathrm{stiff}}$ measures the time for the assembly of a local LOD stiffness matrix contribution for a given coefficient according to~\eqref{eq:bilinlocal}.
We assume that this time does not depend on the given coefficient or the coarse element $T$ considered. 
Further, $n_H$ denotes the number of elements in $\calT_H$ which is of the order $H^{-d}$.

For the standard LOD, the global LOD matrix is newly computed for each sample. Hence, the total time for LOD matrix assemblies amounts to
$
t_{\mathrm{tot}}^s:= M_{\mathrm{samp}} n_H t_{\mathrm{stiff}}.
$
In the LOD with local updates of \cite{HelKM20},
the LOD stiffness matrix for $A_0=A_{\varepsilon}$ is calculated on a single element.
For each sample coefficient $A$, an error indicator is then evaluated for each element $T$ which requires a time of $n_H t_{\mathrm{ind}}$.
For the fraction $p_{\mathrm{recomp}}$ of elements where the error indicator is the largest, the LOD stiffness matrix is computed anew based on $A$. 
Overall, the LOD with local updates takes a total matrix assembly time of
$t_{\mathrm{tot}}^u := t_{\mathrm{stiff}} + M_{\mathrm{samp}}(n_H t_{\mathrm{indic}}+p_{\mathrm{recomp}}n_H t_{\mathrm{stiff}}).$

In the offline-online strategy, we first compute LOD stiffness matrices for $(N+1)$ coefficients, yielding a run-time of $(N+1)t_{\mathrm{stiff}}$ in the offline phase.
We recall that $N$ is of the order $(mH/\varepsilon)^d$.
In the online phase, we denote by $t_{\mathrm{comb}}$ the time to combine the LOD stiffness matrices.
Hence, we obtain the total matrix assembly time for the offline-online strategy as
$t_{\mathrm{tot}}^o:= (N+1)t_{\mathrm{stiff}} +  M_{\mathrm{samp}} n_H t_{\mathrm{comb}}.$

We observe that the offline-online strategy easily outperforms the standard LOD, i.e., $t_\mathrm{tot}^o< t_\mathrm{tot}^s$ because we can expect $t_{\mathrm{comb}} \ll t_{\mathrm{stiff}}$: Forming the linear combination in~\eqref{eq:offlineonline} is much faster than computing the correction operator $\calC_{m,T}(A)$.
The main goal thus is to outperform the LOD with local updates, i.e., to achieve $t_\mathrm{tot}^o< t_\mathrm{tot}^u$.
This requires $t_{\mathrm{comb}}<t_{\mathrm{indic}}+ p_{\mathrm{recomp}}t_{\mathrm{stiff}}$, which is easily achievable in practice, and
we can even hope for $t_{\mathrm{comb}}<t_{\mathrm{indic}}$.
We then deduce that $t_\mathrm{tot}^o< t_\mathrm{tot}^u$ if and only if
$M_{\mathrm{samp}}>\frac{Nt_{\mathrm{stiff}}}{n_H(t_{\mathrm{indic}}+p_{\mathrm{recomp}}t_{\mathrm{stiff}}-t_{\mathrm{comb}})}.$
With the previous comments on the relation of $t_\mathrm{comb}, t_\mathrm{stiff}$ and $t_\mathrm{indic}$, the offline-online strategy will be more efficient than the LOD with local updates already for moderate sample sizes. This holds especially in the regime $H\approx \sqrt{\varepsilon}$ where $N/n_H$ is small.

\subsection{Computation of $E_T$}\label{subsec:indicator}
We finally give some details on the implementation of $E_T$ from Theorem~\ref{thm:errorhigherdim}, where we only consider the case $A=\bar{A}$ for simplicity.
When computing $E_T$ we set up an eigenvalue problem. We denote by $\lambda_j$ the local basis functions of $V_H$ on $T$. There are $d+1$ basis functions for simplicial meshes and $2^d$ on quadrilateral meshes.
Now, $E_T$ is the square root of the largest eigenvalue to the generalized eigenvalue problem
$$
\mathbf{S}v=\nu \mathbf{B}v,
$$
where 
\begin{align*}
S_{jk}&=\Bigl(\sum_{i=0}^N \mu_i(A^{1/2}-A^{-1/2}A_i)\nabla (\calC_{m,T}(A_i) \lambda_k),\sum_{i=0}^N \mu_i(A^{1/2}-A^{-1/2}A_i)\nabla (\calC_{m,T}(A_i) \lambda_j)\Bigr)_{U_m(T)} \\*
B_{jk}&=(A\nabla \lambda_k,\nabla \lambda_j)_T.
\end{align*}

To assemble $\mathbf{S}$, we need to store $\nabla \calC_{m,T}(A_i)\lambda_j$ for $i=0,\dots,N$ and $j=1,\dots,d$ (since $\sum \lambda_j=1$) on each (fine)  element of the patch surrounding $T$ (in case of a simplicial mesh).
Since these gradients are piecewise constant on the fine mesh, this amounts to $(N+1)d$ vectors of length $d$ per fine element.
Note that the number of fine elements in a patch is of the order $(mH/h)^d$.
Hence, additional to the cheap storage of $A_i$ (see Section~\ref{subsec:algo}), quantities of order $Nd^2(mH/h)^d$ have to be stored for computing $E_T$.
We emphasize that $\nabla \calC_{m,T}(A_i)\lambda_j$ only needs to be stored for a single, fixed element $T$ due to periodicity.
For each $A=\sum_{i=0}^N \mu_i A_i$ we then need to compute the component-wise products between $(A^{1/2}-A^{-1/2}A_i)\mu_i$ and the stored $\nabla \calC_{m,T}(A_i)\lambda_j$, sum over $i$ and finally compute the integrals for the combinations of $1\leq j,k\leq d+1$. Here, we can exploit symmetry and the fact that many $\mu_i$ are zero. 
We emphasize that the computation of $\textbf{B}$ is cheaper.

\begin{remark}
	We note once more that $E_T$ is not required in the offline-online strategy, but serves as error control. Hence, instead of evaluating $E_T$ for each sample coefficient, one could also try to estimate the distribution of $E_T$ by sampling (via sampling $\mu_i$).
	This could be done for a single element $T$ in the (extended) offline phase so that $\calC_{m, T}(A_i)$ can be discarded for the online phase as before.
\end{remark}

\section{Numerical experiments}\label{sec:numexp}
In this section, we present extensive numerical examples in one and two dimensions on the unit cell $D=[0,1]^d$.
We choose $f=8\pi^2\sin(2\pi x)$ in the one-dimensional experiment and $f=8\pi^2\sin(2\pi x_1)\cos(2\pi x_2)$ in two dimensions.
Our code is based on \texttt{gridlod} \cite{HelK19} and is publicly available at \url{https://github.com/BarbaraV/gridlod-random-perturbations}.

We consider the following relative errors between the PG-LOD solution $u_m^H$ for the given sample coefficient and the solution $\tilde{u}_m^H$ of our offline-online strategy as well as their upscaled version $u_m^{\mathrm{ms}}$ and $\tilde u_m^{\mathrm{ms}}$, respectively, 
	\begin{equation}\label{eq:relerror}
	\frac{\|u_m^H-\tilde u_m^H\|_{L^2(D)}}{\|u_m^H\|_{L^2(D)}}\qquad\text{and}\qquad
	\frac{|u_m^{\mathrm{ms}}-\tilde u_m^{\mathrm{ms}}|_{H^1(D)}}{|u_m^{\mathrm{ms}}|_{H^1(D)}}.
	\end{equation}
We additionally take the root mean square error over $M_\mathrm{samp}$ samples.

\subsection{One-dimensional example}
We set $h=\varepsilon=2^{-8}$ and randomly assign to each interval of length $\varepsilon$ the value $\alpha=0.1$ with probability $1-p$ or the value $\beta=1.0$ with probability $p$.
We compute the maximal error between the element-wise harmonic means $A_\mathrm{harm}|_T$ and $A_\mathrm{harm}^\mu|_T$, cf.~Section~\ref{sec:analysis}, as well as the relative $L^2(D)$-errors for the solutions (with $m=0$).
The root mean square errors over $500$ samples in dependence on the probability $p$ and on the mesh size $H$ are depicted in Figure~\ref{fig:1d:errors}.
On the left, we see that all curves for the harmonic means show an overall quadratic behavior as expected from Theorem~\ref{thm:error1D}. Moreover, we also see the predicted increase of the constant with decreasing $H$.
Figure~\ref{fig:1d:errors} (right) illustrate the quadratic dependence of the root mean square errors in the solution on $p$. Here, however, the curves for different $H$ lie closer together.
Moreover, we observe that the relative $L^2(D)$-errors (slightly) decrease with decreasing $H$.
Overall, we also observe that, in this setting, root mean square errors are below $3\%$ for probabilities uo to $p=0.2$.

\begin{figure}
	\centering
	\includegraphics[width=0.7\textwidth, trim=5mm 0mm 15mm 10mm, clip=true]{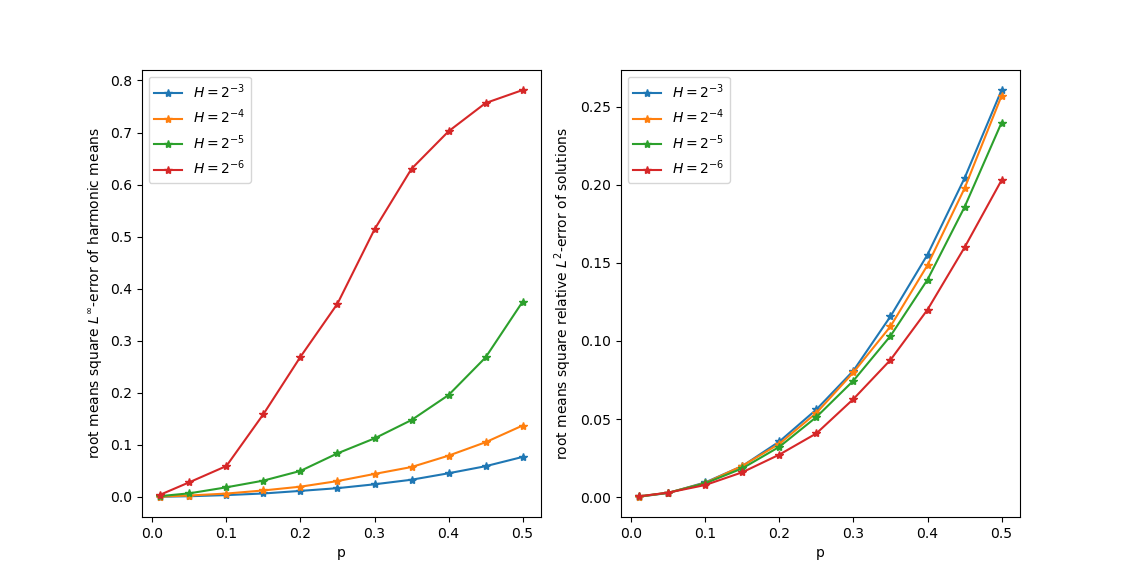}
	\caption{Root mean square errors for harmonic means in $L^\infty(D)$-norm (left) and root mean square relative $L^2(D)$-errors of the solution (right) versus probability for the one-dimensional random checkerboard with $\varepsilon=2^{-8}$ and varying $H$.}
	\label{fig:1d:errors}
\end{figure}

\subsection{Random checkerboard coefficient}\label{subsec:numexp:randcheck2d}
We now consider a two-dimensional random checkerboard coefficient, i.e., we set $d=2$ and consider Example~\ref{ex:randcheck} with $\varepsilon=2^{-7}$, $\alpha=0.1$ and $\beta=1.0$.
For the LOD we set $h=2^{-8}$, so that all fine-scale details are resolved, as well as $H=2^{-5}$ and $m=4$.
The choice of $H$ and $m$ ensures that $u_m^H$ of the standard LOD is a good reference solution.
All the following results are obtained with $M_\mathrm{samp}=250$.
First, we investigate the behavior of the root mean square relative $L^2(D)$- and $H^1(D)$-error of our new approach in dependence on the probability $p$.
We observe in Figure~\ref{fig:2d:errorsvsprob} (left) that the root mean square $L^2$-error stays below roughly $3\%$ and the $H^1$-error below about $10\%$ for probabilities up to $p=0.1$ in this example. Further, both errors grow rather moderately with $p$. Both observations indicate the good performance of our new approach for this example.
In contrast, simply computing the LOD solution for $A_\varepsilon$ -- which is deterministic -- gives a rather poor approximation as expected, see Figure~\ref{fig:2d:errorsvsprob}, right.
More precisely, both root mean square errors are an order of magnitude higher than for the new approach. In detail, the $L^2$-error is already almost $2\%$ and the $H^1$-error about $12\%$ for $p=0.01$ in this example.
Moreover, a closer inspection of the data indicates that the root mean square errors grow faster in the right figure than in the left one: For the offline-online strategy the $L^2$-errors seem to grow as $p^{3/2}$ and the $H^1$-errors as $p$ in comparison to a growth like $p$ for the $L^2$-errors and of $p^{1/2}$ for the $H^1$-errors for the LOD solution with $A_\varepsilon$.
A detailed investigation of the exact dependence of the errors on $p$ for the offline-online strategy in higher dimensions is beyond the scope of this article.  
This also implies that  the LOD with local updates from \cite{HelKM20} needs an increasing fraction of basis updates. For instance, for $p=0.1$, we would require about $60\%$ updates in this example to attain an accuracy comparable to the offline-online strategy. 
Put differently, for $p=0.1$, the LOD with $15\%$ updates (which is a reasonable amount and considered for the timings below) leads to a root mean square error of about $17\%$, in contrast to the reported $3\%$ for the offline-online strategy.

\begin{figure}
	\centering
	\includegraphics[width=0.7\textwidth, trim=5mm 0mm 15mm 10mm, clip=true]{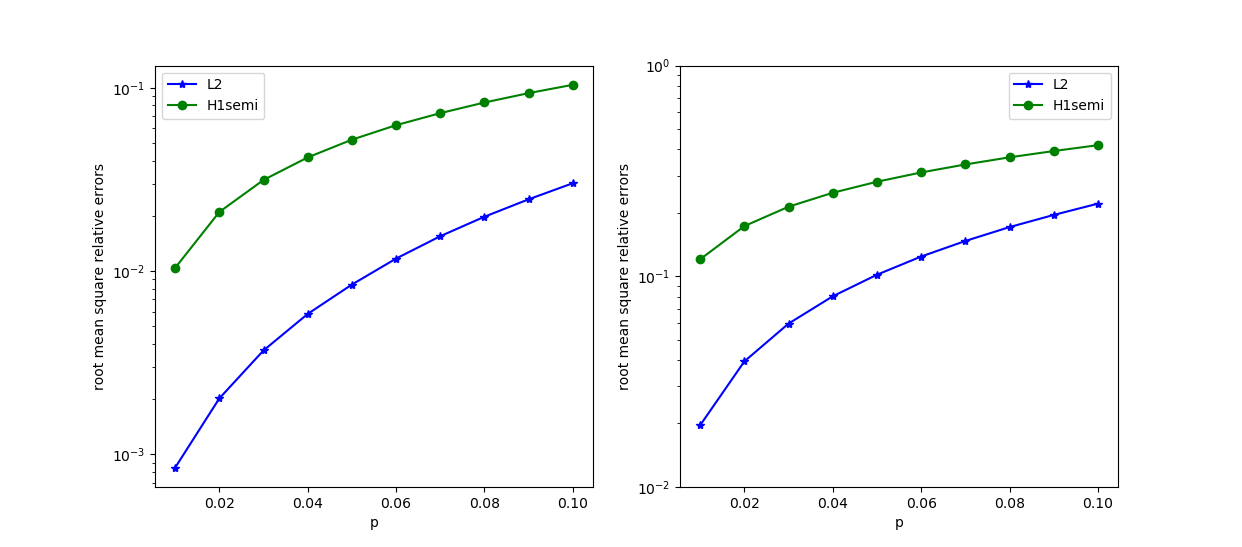}%
	\caption{Root mean square relative errors versus probability for the two-dimensional random checkerboard. Left:  offline-online strategy, right: LOD solution for deterministic coefficient~$A_\varepsilon$.}
	\label{fig:2d:errorsvsprob}
\end{figure}

In the spirit of Section~\ref{subsec:complexity}, let us briefly comment on the run-time of our offline-online-strategy. We consider the same setting as before and fix $p=0.1$. Timings were taken without the parallelization possibilities discussed in Section~\ref{subsec:algo} on a standard desktop computer (Intel i5-7500 core, frequency 3.40 GHz, Ubuntu 20.04).
With the proposed strategy, the offline phase takes about $186$ seconds and the assembly of the \emph{global} stiffness matrix takes less than $3$ seconds for a single sample coefficient in the online phase (averaged over $250$ samples).
In contrast, computing a new LOD stiffness matrix completely for each coefficient takes about $145$ seconds per sample (averaged over $250$ samples). This clearly shows that a standard LOD becomes extremely costly in many Monte Carlo settings.
Computing the LOD stiffness matrix on a single element for $A_\varepsilon$ takes about $0.1$ seconds. The LOD with $15\%$ updates requires about $23$ seconds as matrix assembly time -- including evaluation of the corrector and local re-computations -- for a single sample (averaged over $250$ samples).
This indicates that the offline-online strategy is more attractive than local updates already for moderate sample sizes -- in the considered example,  after about $10$ samples.

\subsection{Periodic coefficient with random defects}
We now consider the two-dimensional version of Example~\ref{ex:randdef} with $\varepsilon=2^{-6}$, $\alpha=1$ and $\beta=10$.
For the PG-LOD, we select $h=2^{-8}$, $H=2^{-4}$ and $m=3$ again guaranteeing that all inclusions are resolved by the fine mesh and that the PG-LOD solution $u_m^H$ serves as a good reference.
We consider the root mean square relative $L^2(D)$-errors for our offline-online strategy over $350$ samples and for defect probabilities $p\in \{0.01, 0.05, 0.1, 0.15\}$.
\begin{figure}
	\centering
	\includegraphics[width=\textwidth, trim=20mm 7mm 18mm 12mm, clip=true]{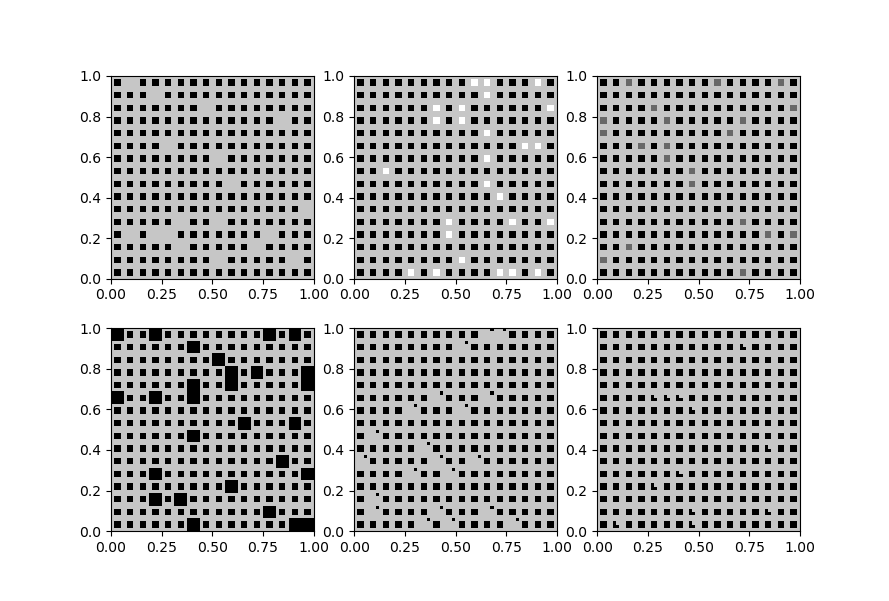}%
	\caption{Different variants of random defects in periodic multiscale coeffcients for $\varepsilon=2^{-4}$ and $p=0.1$. Top row: Changes in value with $\tilde{\beta}\in\{1., 0.5, 5\}$ (from left to right). Bottom row from left to right: \texttt{fill}, \texttt{shift}, and \texttt{Lshape}.}
	\label{fig:def:variants}
\end{figure}
Here we focus on the influence of the following defect possibilities on the errors:
\begin{itemize}
	\item A defect inclusion has the new value $\tilde{\beta}\in \{1, 0.5, 5\}$ where we emphasize that $\tilde{\beta}=1$ means that the inclusion vanishes;
	\item A defect inclusion takes the value $\beta$ in the whole $\varepsilon$-cell (called \texttt{fill});
	\item A defect inclusion is positioned at (scaled and shifted versions of) $[0.75, 1]^2$ (called \texttt{shift});
	\item A defect inclusion has a different shape of (scaled and shifted versions of) $[0.25, 0.75]^2\setminus [0.5, 0.75]^2$ (called \texttt{Lshape}).
\end{itemize}
The different considered possibilities for $p=0.1$ are visualized in Figure~\ref{fig:def:variants} where we chose $\varepsilon=2^{-4}$ to better see the (fine) inclusions.
Note that the model \texttt{shift} does not only shift the inclusion, but even changes its size.

The root mean square relative $L^2(D)$-errors of our offline-online strategy depending on $p$ are depicted in Figure~\ref{fig:2d:defects}.
On the left, we see the influence of the value taken in a defect. 
While all root mean square errors are very small (below $0.5\%$),  the smaller $\tilde{\beta}$, i.e., the value in the defect, the larger the error.
The dependency of the error on the contrast between $\alpha$, $\beta$, and $\tilde{\beta}$ is also indicated by our theoretical findings.

\begin{figure}
	\includegraphics[width=0.47\textwidth, trim=0mm 0mm 15mm 10mm, clip=true]{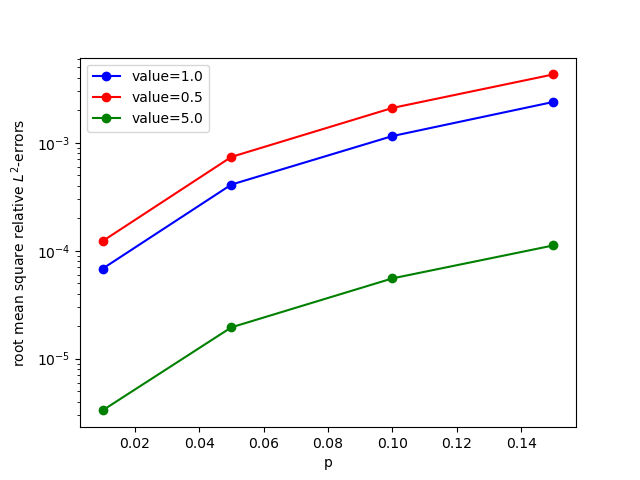}%
	\hspace{2ex}
	\includegraphics[width=0.47\textwidth, trim=0mm 0mm 15mm 10mm, clip=true]{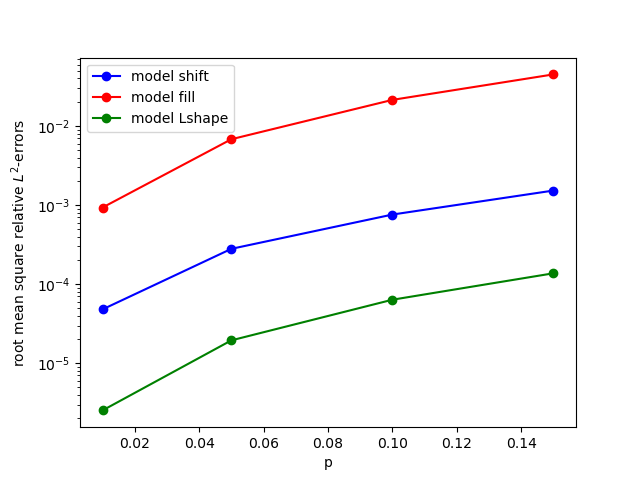}%
	\caption{Root mean square relative $L^2(D)$-error versus probability for the two-dimensional periodic inclusions with defects. Comparison of different defect values (left) and geometry changes (right).}
	\label{fig:2d:defects}
\end{figure}

Figure~\ref{fig:2d:defects} (right) shows the influence of the different geometry changes in the defects on the root mean square errors.
In this experiment, we have errors below $5\%$ up to $p=0.15$ for all considered configurations.
Concerning the differences between the geometry changes, we clearly observe that the model \texttt{fill} is the most difficult and produces the largest errors.
This holds true not only in comparison to \texttt{shift} and \texttt{Lshape} depicted in Figure~\ref{fig:2d:defects} right, but also in comparison to \texttt{value} depicted in Figure~\ref{fig:2d:defects} (left).

\subsection{Error indicator $E_T$}\label{subsec:numexp:ET}
The aim of our final numerical experiment is to illustrate the relation between the error $u_m^H-\tilde{u}_m^H$ locally on an element patch $U_m(T)$ and the indicator $E_T$. 
Note that this is not exactly covered by the theoretical findings of Theorem~\ref{thm:errorhigherdim}.
We approach the question by studying the random checkerboard coefficient of Example~\ref{ex:randcheck} in $d=2$ with $\alpha=0.1$ and $\beta=1.0$. We set $H=1/5$, $m=2$, $\varepsilon=1/20$ and $h=1/40$ so that $D=[0, 1]^2$ represents the element patch $U_m(T)$.
We emphasize that $D$ in this experiment is not meant to represent an actual full computational domain, but serves as proxy for the patch $U_m(T)$. For a ``full'' experiment with our approach, the computational domain would be a domain $\tilde D$ with (approximately at least) $\operatorname{diam}( \tilde D)\gtrsim 6$ in the above setting. However, since we are only interested on the behavior on a single patch $U_m(T)$ in this section, we restrict also our computations to this patch.

For each sample coefficient, we compute $u_m^H$ and $\tilde{u}_m^H$ as in the previous experiments. In this experiment, we consider both the root mean square absolute and relative $L^2(D)$-error as well as the root mean square relative $H^1(D)$-error. 
Additionally, we compute $E_T$ for the element $T$ in the middle of $D$.
Figure~\ref{fig:indic} depicts the root mean squares -- sampled over 500 realizations -- of the errors and the indicator for different values of $p$. We see that the absolute $L^2(D)$-error and $E_T$ show a qualitatively and quantitatively similar behavior which underlines the validity of $E_T$ as an indicator for the (local) error.
The relative $L^2(D)$-error is of a different magnitude since we divide by the norm of $u_m^H$, cf.~\eqref{eq:relerror}.
Amplifying the relative error by a factor of $4$, we observe that there seems to be a good qualitative agreement between the relative error and the indicator (Figure~\ref{fig:indic}, dashed line). In other words, the ratio between $E_T$ and the absolute as well as the relative $L^2(D)$-error seems to be almost constant for varying probabilities $p$.
Note that the almost constant factor between the root mean square of the absolute and relative error also indicates that $\|u_{m}^H\|_{L^2(D)}$ is (almost) independent of $p$ and close to $4$ in this example.
Further, we observe in Figure~\ref{fig:indic} also a good qualitative agreement between the error indicator and the relative $H^1(D)$-error (multiplied by 2). Hence, we can summarize that the previously discussed results for the relation between $E_T$ and the $L^2(D)$-errors seems to carry over to the $H^1(D)$-errors as well. Note that we omitted the absolute $H^1(D)$-errors for better visibility.

\begin{figure}
	\centering
	\includegraphics[width=0.6\textwidth, trim=3mm 0mm 15mm 10mm, clip=true]{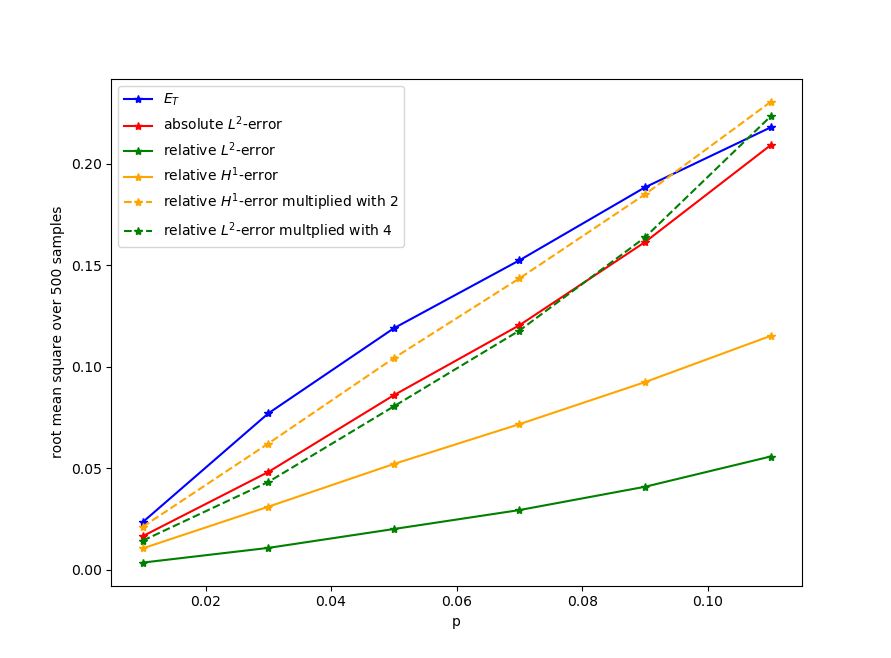}
	\caption{Root mean square of relative and absolute $L^2(D)$-errors, relative $H^1(D)$-error, and indicator $E_T$ versus probability for the two-dimensional random checkerboard.}
	\label{fig:indic}
\end{figure}

\section*{Conclusion}
We presented an offline-online strategy based on the Localized Orthogonal Decomposition (LOD) to compute coarse-scale solutions of elliptic equations with local random perturbations in a Monte Carlo setting.
Exploiting the periodic structure of the underlying deterministic coefficient, LOD stiffness matrices on a single reference element for representative perturbations are computed in an offline phase.
Those are efficiently combined to yield the LOD stiffness matrix for each sample coefficient in the online phase.
We showed error estimates in the one- as well as the higher dimensional case where we derived an indicator. 
Theoretical algorithmic considerations as well as numerical experiments underlined the promising performance of the offline-online strategy.
Overall, the present contribution provides interesting first results on an efficient computational multiscale approach for PDEs with random coefficients. Future research concerns the combination with Monte Carlo-type approaches, the extension of the construction to more general random coefficients, and the improvement of the error indicator to reduce storage of fine-scale quantities.

\appendix
\section{Proof of Theorem~\ref{thm:error1D}}\label{sec:appendix}

For the proof of Theorem~\ref{thm:error1D}, we use the notation from Section~\ref{sec:analysis} and introduce some further abbreviations.
We recall the coefficient $A_\mathrm{harm}^\mu|_T$ associated with $\tilde{\frakb}_T$ for given $\mu_i$  with $\sum_{i=0}^N\mu_i=1$, see~\eqref{eq:Aharmmu}.
Further, we introduce
\begin{equation}\label{eq:barA}
\overline{A}:=\int_0^\varepsilon\frac{1}{A_\varepsilon}\, dx
\end{equation}
and point out that we have
\[\int_T\frac{1}{A_\varepsilon}\, dx = N \overline{A}\]
by the periodicity of $A_\varepsilon$.
Similarly, we denote 
\begin{equation}\label{eq:barApert}
\overline{A}_\mathrm{def}=\int_{\varepsilon(k+Q)}\frac{1}{A_\varepsilon+B_\varepsilon}-\frac{1}{A_\varepsilon}\, dx
\end{equation}
and emphasize that $\overline{A}_\mathrm{def}$ is independent of the choice of $k\in I$ with the index set $I$ from~\eqref{eq:randomness}.

\begin{proof}[Proof of Theorem~\ref{thm:error1D}]
	We estimate the consistency error as
	\begin{align*}
	|(\frakb-\tilde{\frakb})(v,w)|&\leq \sum_{T\in \calT_H}|(\frakb_T-\tilde{\frakb}_T)(v,w)|\\
	&\leq \sum_{T\in \calT_H}\big|A_{\mathrm{harm}}|_T-A_\mathrm{harm}^\mu|_T\big|\,\, \|\nabla v\|_{L^2(T)}\,\|\nabla w\|_{L^2(T)}\\
	&\leq \alpha^{-1} \bigl(\max_{T\in \calT_H} \big|A_\mathrm{harm}|_T-A_\mathrm{harm}^\mu|_T\big|\bigr)\,  \|v\|_{A}\|w\|_{A}.
	\end{align*}
	Hence, it suffices to estimate $|A_\mathrm{harm}|_T-A_\mathrm{harm}^\mu|_T|$ for any $T\in \calT_H$.
	
	With the preliminary observations and notation from above, we have
	\begin{align*}
	A_\mathrm{harm}|_T &= \frac{|T|}{N\overline{A}+N_\mathrm{def}\overline{A}_\mathrm{def}},\\*
	A_\mathrm{harm}^0|_T &= \frac{|T|}{N\overline{A}},\qquad \text{and}\qquad A_\mathrm{harm}^i|_T =\frac{|T|}{N\overline{A}+\overline{A}_\mathrm{def}} \qquad \text{for}\quad i=1,\ldots, N.
	\end{align*}
	Recall that for $i=1, \ldots, N$ we have $\mu_i\in\{0,1\}$ and $\mu_0=1-N_\mathrm{def}$.
	Hence, we deduce
	\begin{equation}\label{eq:Aharmmu1}
	\begin{aligned}
	A_\mathrm{harm}^\mu|_T&=\sum \mu_i A_\mathrm{harm}^i|_T=(1-N_\mathrm{def})\frac{|T|}{N\overline{A}} + N_\mathrm{def}\frac{|T|}{N\overline{A}+\overline{A}_\mathrm{def}}\\
	&=\frac{N\overline{A} |T|+(1-N_\mathrm{def})\overline{A}_\mathrm{def}|T|}{N\overline{A} (N\overline{A}+\overline{A}_\mathrm{def})}\\
	&=\frac{|T|}{N\overline{A}}-\frac{N_\mathrm{def}\overline{A}_\mathrm{def}|T|}{N\overline{A} (N\overline{A}+\overline{A}_\mathrm{def})}.
	\end{aligned}
	\end{equation}
	Combining~\eqref{eq:Aharmmu1} with the expression for $A_\mathrm{harm}|_T$, inserting $N_\mathrm{def}=\theta_{\mathrm{def}, T}N$, and performing a Taylor expansion around $\theta=0$, we obtain
	\begin{align*}
	|A_\mathrm{harm}|_T-A_\mathrm{harm}^\mu|_T|&=|T|\, \Big|\frac{1}{N\overline{A}+\theta_\mathrm{def, T}N\, \overline{A}_\mathrm{def}} - \frac{1}{N\overline{A}} + \theta_{\mathrm{def}, T}\frac{\overline{A}_\mathrm{def}}{\overline{A}(N\overline{A}+\overline{A}_\mathrm{def})}\Bigr|\\
	&\leq |T| \Bigl(0+\theta_{\mathrm{def}, T} \Bigl|\frac{\overline{A}_\mathrm{def}}{\overline{A}(N\overline{A}+\overline{A}_\mathrm{def})}-\frac{\overline{A}_\mathrm{def}}{N\overline{A}^2}\Bigr| + \theta_{\mathrm{def}, T}^2\Bigr|\frac{2\overline{A}_\mathrm{def}^2}{N(\overline{A}+\eta \overline{A}_\mathrm{def})^3}\Bigr|\Bigr)
	\end{align*}
	for some $\eta\in [0, \theta_{\mathrm{def, T}}]$.
	
	Before we estimate the first- and second-order term in the expansion separately, we bound $\overline{A}$ and $\overline{A}_\mathrm{def}$ using~\eqref{eq:boundsAper}--\eqref{eq:boundsAperCper}. We obtain
	\begin{align*}
	\frac{\varepsilon}{\beta}\leq\overline{A}\leq \frac{\varepsilon}{\alpha}\qquad \text{and} \qquad  \overline{A}_\mathrm{def}\leq\varepsilon|Q|\Bigl(\frac{1}{\alpha}-\frac{1}{\beta}\Bigr)
	\end{align*}
	as well as (writing $Q=[q_0, q_1]$)
	\begin{align*}
	N\overline{A}+\overline{A}_\mathrm{def}&=N\int_0^\varepsilon \frac{1}{A_\varepsilon}\, dx+\int_{\varepsilon q_0}^{\varepsilon q_1}\frac{1}{A_\varepsilon+B_\varepsilon}-\frac{1}{A_\varepsilon}\, dx\\
	&=(N-1)\int_0^\varepsilon \frac{1}{A_\varepsilon}\, dx + \int_0^{\varepsilon q_0} \frac{1}{A_\varepsilon}\, dx + \int_{\varepsilon q_1}^\varepsilon \frac{1}{A_\varepsilon}\, dx + \int_{\varepsilon q_0}^{\varepsilon q_1}\frac{1}{A_\varepsilon+B_\varepsilon}\, dx\\
	&\geq \frac{1}{\beta}((N-1)\varepsilon+\varepsilon q_0+\varepsilon(1-q_1)+\varepsilon(q_1-q_0))=\frac{N\varepsilon}{\beta}.
	\end{align*}
	
	For the first-order term in the Taylor expansion we deduce
	\begin{align*}
	\Bigl|\frac{\overline{A}_\mathrm{def}}{\overline{A}(N\overline{A}+\overline{A}_\mathrm{def})}-\frac{\overline{A}_\mathrm{def}}{N\overline{A}^2}\Bigr|&=\Bigl|\frac{\overline{A}_\mathrm{def}(N\overline{A}-(N\overline{A}+\overline{A}_\mathrm{def}))}{N\overline{A}^2(N\overline{A}+\overline{A}_\mathrm{def})}\Bigr|=\Bigl|\frac{\overline{A}_\mathrm{def}^2}{N\overline{A}^2(N\overline{A}+\overline{A}_\mathrm{def})}\Bigr|\\*
	&\leq \varepsilon^2 |Q|^2\Bigl(\frac{1}{\alpha}-\frac{1}{\beta}\Bigr)^2 \frac{\beta^3}{\varepsilon^3 N^2}=\frac{\beta^3\, |Q|^2}{\varepsilon\, N^2} \Bigl(\frac{1}{\alpha}-\frac{1}{\beta}\Bigr)^2.
	\end{align*}
	Similarly, the second-order term in the Taylor expansion can be estimated as
	\begin{align*}
	\Bigr|\frac{2\overline{A}_\mathrm{def}^2}{N(\overline{A}+\eta \overline{A}_\mathrm{def})^3}\Bigr|&\leq \frac{2\varepsilon^2|Q|^2}{N}\Bigl(\frac{1}{\alpha}-\frac{1}{\beta}\Bigr)^2\frac{\beta^3}{\varepsilon^3 (1+|Q|(1-\eta))^3}\\*
	&\leq \frac{2\varepsilon^2|Q|^2\beta^3}{N\varepsilon}\Bigl(\frac{1}{\alpha}-\frac{1}{\beta}\Bigr)^2.
	\end{align*}
	Finally, we obtain with $|T|=H$ and $N=H/\varepsilon$ that
	\begin{equation*}
	|A_\mathrm{harm}|_T-A_\mathrm{harm}^\mu|_T|\leq |Q|^2 \beta^3 \Bigl(\frac{1}{\alpha}-\frac{1}{\beta}\Bigr)^2 \frac{\varepsilon}{H}\theta_{\mathrm{def}, T}+ 2|Q|^2\beta^3 \Bigl(\frac{1}{\alpha}-\frac{1}{\beta}\Bigr)^2 \theta_{\mathrm{def}, T}^2.\qedhere
	\end{equation*}
\end{proof}

\end{document}